\documentclass[11pt]{amsart}
\usepackage{amsmath, amsthm, amssymb,amscd}
\usepackage{float}
\usepackage{graphicx}
\usepackage[arrow, matrix, curve]{xy}
\usepackage{color}
\usepackage{enumitem}  
\usepackage{tikz-cd} 
\usepackage{url}

\usepackage{blindtext}

\usepackage[english]{babel}
\newcommand{\seq}{\subseteq}
\newcommand{\C}{\mathbb{C}}

\newtheorem{thm}{Theorem}[section]
\newtheorem*{thm-nl}{Theorem}
\newtheorem*{prop-nl}{Proposition}

\newtheorem{lem}[thm]{Lemma}

\def\PP{{\mathbf P}}

\def\Pic0{{\mathrm Pic}^0(X)}

\newtheorem{cor}[thm]{Corollary}
\newtheorem*{cor-nl}{Corollary}

\newtheorem*{conjecture-nl}{Conjecture}

\newtheorem*{quest-nl}{Question}
\newtheorem*{quests-nl}{Questions}

\newtheorem{prop}[thm]{Proposition}

\theoremstyle{remark}
\newtheorem*{rem}{Remark}

\newtheorem{remark}[thm]{Remark}

\title{{Universal Secant Bundles and Syzygies of Canonical Curves}}
\date{\today}

\author[M. Kemeny]{Michael Kemeny}

\address{University of Wisconsin-Madison, Department of Mathematics, 480 Lincoln Dr
\hfill \newline\texttt{}
 \indent WI 53706, USA} \email{{\tt michael.kemeny@gmail.com}}

\begin{document}
\begin{abstract}
We introduce a relativization of the secant sheaves from \cite{green-laz-petri} and \cite{ein-lazarsfeld-asymptotic} and apply this construction to the study of syzygies of canonical curves. As a first application, we give a simpler proof of Voisin's Theorem for general canonical curves. This completely determines the terms of the minimal free resolution of the coordinate ring of such curves. Secondly, in the case of curves of even genus, we enhance Voisin's Theorem by providing a structure theorem for the last syzygy space, resolving the Geometric Syzygy Conjecture in even genus.
\end{abstract}
\maketitle
\setcounter{section}{-1}
\section{Introduction}
In this paper, we introduce a \emph{universal} version of the secant sheaf construction from \cite{ein-lazarsfeld-asymptotic}. Using this tool, we give a simpler proof of a theorem of Voisin \cite{V1}, \cite{V2} on the equations of canonical curves. We further give a generalization of her result for curves of even genus.\smallskip

 The classical Theorem of Noether--Babbage--Petri states that canonical curves are projectively normal, and that the ideal $I_{C/\PP^{g-1}}$ is generated by quadrics (with a few exceptions), see \cite{arbarello-sernesi-petri} for a modern treatment. In the 1980s, M.\ Green realized that these classical results about the equations defining canonical curves should be the first case of a much more general statement about higher syzygies, and he made a very influential conjecture \cite{green-koszul} in this direction.\smallskip

Whilst the general case of Green's Conjecture remains open, in 2002 Voisin made a breakthrough by proving the conjecture for \emph{general} curves of even genus \cite{V1}.  Voisin's argument relies on an intricate study of the geometry of Hilbert schemes on a K3 surface. Recently, an algebraic approach to Voisin's Theorem has been given, \cite{AFPRW}, based on degenerating to the tangent developable, a singular surface whose hyperplane sections are cuspidal curves. The authors apply the representation theory of an $SL_2$ action present in this special situation to establish Green's conjecture for rational cuspidal curves. Explicit plethysm formulae play the key role, involving a change of basis between elementary symmetric polynomials and Schur polynomials. Maps which are simple to describe in one basis become rather complicated in the other, making the proof quite technical, see \cite[\S 5.5--5.7]{AFPRW}. \smallskip


In this paper, we first give a simpler proof of Voisin's Theorem, using only basic homological algebra and without the need to degenerate. We further provide a structure theorem in the even genus case, describing in detail the extremal syzygy space. Let $X$ be a complex K3 surface with Picard group generated by an ample line bundle $L$ of even genus $g=2k$, i.e.\ $(L)^2=2g-2$. We proceed by direct computation on $X$. Define $\mathrm{K}_{p,q}(X,L)$ as the middle cohomology of 
{\small{$$\bigwedge^{p+1} \mathrm{H}^0(X,L) \otimes \mathrm{H}^0(X,L^{\otimes q-1}) \to \bigwedge^p \mathrm{H}^0(X,L) \otimes \mathrm{H}^0(X,L^{\otimes q}) \to \bigwedge^{p-1} \mathrm{H}^0(X,L) \otimes \mathrm{H}^0(X,L^{\otimes q+1}) $$}}
 Voisin's Theorem states that $\mathrm{K}_{k,1}(X,L)=0$, \cite{V1}. This single vanishing suffices to prove Green's Conjecture for general canonical curves in even genus. \smallskip

Let $E$ be the rank two \emph{Lazarsfeld--Mukai bundle} associated to a $g^1_{k+1}$ on a smooth curve $C \in |L|$, see \cite{lazarsfeld-BNP}. The dual bundle $E^{\vee}$ fits into the exact sequence
$$0 \to E^{\vee} \to \mathrm{H}^0(C,A) \otimes \mathcal{O}_X \to i_*A \to 0,$$
for $A$ a $g^1_{k+1}$ on $C$, where $i: C \hookrightarrow X$ is the inclusion. The vector bundle $E$ has invariants $\det(E)=L$, $h^0(E)=k+2$, $h^1(E)=h^2(E)=0$.\smallskip

We deduce Voisin's Theorem from the K\"unneth formula on $X \times \PP(\mathrm{H}^0(E))$. Our proof reduces to showing that a certain square matrix is nonsingular. Since our matrix takes the form $\mathrm{H}^k(\mathrm{Sym}^{k+1} \mathcal{G}) \to \mathrm{H}^k(\mathrm{Sym}^{k} \mathcal{G} \otimes \mathcal{G})$, for some bundle $\mathcal{G}$, the desired nonsingularity is \emph{automatic}, see the proof of Proposition \ref{auto-injectivity}.\smallskip

The starting point of Voisin's proof \cite{V1} is her \emph{Hilbert Scheme Description} of Koszul cohomology via the cohomology of tautological sheaves on the Hilbert scheme, see \cite[\S 5]{aprodu-nagel}. This description has also been used to effect in \cite{ein-lazarsfeld-gonality}. Our approach uses instead the, considerably simpler and more general, \emph{Kernel Bundle} approach to syzygies, \cite{lazarsfeld-VBT}, as our starting point.\smallskip

The projective space $\PP:=\PP(\mathrm{H}^0(E))$ can be embedded into the Hilbert scheme $X^{[k+1]}$ by sending $s\in \mathrm{H}^0(E)$ to the zero scheme $Z(s)$. Whilst $\PP \seq X^{[k+1]}$ is also a crucial object in Voisin's proof, in order to make her argument work $\PP$ needs to be replaced by the space of zero cycles of the form $Z(s)-x+y$ with $x \in Z(s)$ and $y \in X$. This results in significant complications in Voisin's argument, see \cite[\S 6.3]{aprodu-nagel}. By contrast, we find a way to work directly on $\PP$.\smallskip

  A few years after her breakthrough for even genus curves, Voisin deduced the odd genus case of generic Green's conjecture out of the even genus case \cite{V2}. In Section \ref{odd-genus}, we give a streamlined version of Voisin's argument:
  \begin{thm} \label{generic-green}
  Let $C$ be a general curve over $\C$ of genus $g=2k$ or $g=2k+1$. Then $\mathrm{K}_{k,1}(C,\omega_C)=0$.
  \end{thm}
\smallskip

 Whilst we use several of the same objects from \cite{V2}, such as nodal K3 surfaces, our proof ends us being more economical. In particular, we do away with the explicit computations of \cite[\S 2]{V2} as well as the difficult analysis of the geometry of the Grassmannian from \cite[\S 3, ``Fourth Step'']{V2}, considered by Voisin to be the most essential part of her proof. Furthermore, the application of our techniques to nodal K3 surfaces of even genus is far simpler in our setting.\smallskip

In fact, our approach proves much more than Voisin's theorem for even genus curves. Let $(X,L)$ be any polarized variety. A natural question, with roots going back to Andreotti--Mayer's 1967 paper \cite{andreotti-mayer}, is whether one can find a spanning set of the spaces $\mathrm{K}_{p,1}(X,L)$ consisting of elements of low \emph{rank}. Here one defines the rank of a syzygy $\alpha \in \mathrm{K}_{p,1}(X,L)$ as the dimension of the minimal subspace
$V \seq \mathrm{H}^0(X,L)$ such that $\alpha \in \mathrm{K}_{p,1}(X,L,V)$, i.e.\ such that $\alpha$ is represented by an element of $\bigwedge^p V \otimes \mathrm{H}^0(X,L)$, \cite{bothmer-JPAA}. Syzygies of low rank have geometric meaning. For $\alpha \in \mathrm{K}_{p,1}(X,L)$ we always have $\rm{rank}(\alpha) \geq p+1$. If there exists a syzygy $\alpha$ with $\rm{rank}(\alpha)=p+1$, then $X$ lies on a rational normal scroll. More precisely, the \emph{syzygy scheme} $$\rm{Syz}(\alpha) \seq \PP(\mathrm{H}^0(L)^{\vee})$$ of $\alpha$, as defined by Green \cite{green-canonical}, defines a scroll. Similarly, syzygies of rank $p+2$ arise from linear sections of Grassmannians. Precisely, $\rm{Syz}(\alpha)$ contains the cone over a section of $\rm{Gr}_2(V^{\vee} \oplus \C)$. \smallskip

In Section \ref{sec2}, we prove the following structure theorem for the last nonzero syzygy space for K3 surfaces of even genus.
\begin{thm} \label{geo-thm-intro}
Let $X$ be a complex K3 surface. Assume $\mathrm{Pic}(X)=\mathbb{Z}[L]$ with $L$ ample of even genus $g=2k$. Let $E$ be the Lazarsfeld-Mukai bundle as above. For any nonzero $s \in \mathrm{H}^0(E)$, the space $\mathrm{K}_{k-1,1}(X,L,\mathrm{H}^0(L \otimes I_{Z(s)}))$ is a one-dimensional subspace of $\mathrm{K}_{k-1,1}(X,L)$.
The morphism
{\small{\begin{align*}
\psi \; : \; \PP(\mathrm{H}^0(E)) & \to \PP(\mathrm{K}_{k-1,1}(X,L))\\
[s] & \mapsto [\mathrm{K}_{k-1,1}(X,L,\mathrm{H}^0(L \otimes I_{Z(s)}))]
\end{align*}}}
is the Veronese embedding of degree $k-2$. In particular, $\psi$ induces a natural isomorphism $\mathrm{Sym}^{k-2}\mathrm{H}^0(X,E) \simeq \mathrm{K}_{k-1,1}(X,L) $.
\end{thm}
A consequence of Theorem \ref{geo-thm-intro} is that $\mathrm{K}_{k-1,1}(X,L)$ is spanned by syzygies of rank $k+1=\dim \mathrm{H}^0(L \otimes I_{Z(s)})$. Theorem \ref{geo-thm-intro} implies the $\mathrm{K}_{k,1}(X,L)=0$ by standard dimension computations, \cite[\S 4.1]{farkas-progress}, and thus enhances Voisin's Theorem in even genus. One may compare Theorem \ref{geo-thm-intro} to Schreyer's Conjecture, proven in \cite{lin-syz}, describing the structure of the last syzygy space for curves of \emph{non-maximal} gonality.\smallskip


Theorem \ref{geo-thm-intro} implies a previously open conjecture known as the \emph{Geometric Syzygy Conjecture} in even genus, see \cite{bothmer-Transactions} where the statement is proven for $g \leq 8$.  To put this conjecture in context, recall the following important result:
 \begin{thm}[\cite{andreotti-mayer}, \cite{green-quadrics}] \label{green-quads}
 The ideal $I_{C/\PP^{g-1}}$ of a canonical curve of Clifford index at least two is generated by quadrics of rank four.
 \end{thm}
 This provides an enhancement of Petri's theorem stating that $I_{C/\PP^{g-1}}$ is generated by quadrics if $\mathrm{Cliff}(C) \geq 2$. More precisely, let $W^1_{g-1}(C)$ be the locus of line bundles $A$ of degree $g-1$ with two sections; geometrically this is the locus of double points of the Theta Divisor $\Theta \seq \mathrm{Pic}^{g-1}(C)$. Then $I_{C/\PP^{g-1}}$ is generated by the quadrics defined by the Petri map
 $$\mathrm{H}^0(A) \otimes \mathrm{H}^0(\omega_C \otimes A^{-1}) \to \mathrm{H}^0(\omega_C)$$
 for $A \in W^1_{g-1}(C)$. To extend Theorem \ref{green-quads} to higher syzygies, note that Theorem \ref{green-quads} can be rephrased as stateing that $\mathrm{K}_{1,1}(C,\omega_C)$ is spanned by the spaces $\mathrm{K}_{1,1}(C,\omega_C, \mathrm{H}^0(\omega_C \otimes A^{-1} ))$ for $A \in W^1_{g-1}(C)$, and hence by syzygies of rank \emph{two}.\smallskip

If we restrict the subspaces $\mathrm{K}_{k-1,1}(X,L,\mathrm{H}^0(L \otimes I_{Z(s)})) \seq \mathrm{K}_{k-1,1}(X,L)$ to a curve $C \in |L|$ containing the locus $Z(s)$, then $\mathrm{K}_{k-1,1}(X,L,\mathrm{H}^0(L \otimes I_{Z(s)}))$ restricts to 
$$\mathrm{K}_{k-1,1}(C,\omega_C,\mathrm{H}^0(\omega_C \otimes A^{-1})) \seq \mathrm{K}_{k-1,1}(C,\omega_C)$$ where $A \in W^1_{k+1}(C)$ is the line bundle defined by the divisor $Z(s) \seq C$. Thus, the rank $k+1$ syzygies $\alpha \in \mathrm{K}_{k-1,1}(X,L,\mathrm{H}^0(L \otimes I_{Z(s)}))$ drop rank by one when restricted to $C$. Combining this with an argument of Voisin,  see \cite[Prop.\ 7]{V1} and the unpublished \cite[\S 11]{bothmer-preprint} we obtain:
\begin{cor}[Geometric Syzygy Conjecture in Even Genus] \label{geo-cor}
Let $C$ be a general curve of even genus $g=2k$. Then $\mathrm{K}_{k-1,1}(C,\omega_C)$ is generated by syzygies of the lowest possible rank $k$.  More precisely, $\mathrm{K}_{k-1,1}(C,\omega_C)$ is generated by the rank $k$ syzygies $$\alpha \in \mathrm{K}_{k-1,1}(C,\omega_C,\mathrm{H}^0(\omega_C \otimes A^{-1})), \; \; A \in W^1_{k+1}(C).$$
\end{cor} 
Corollary \ref{geo-cor} therefore provides an extension of Green's theorem on quadrics \cite{green-quadrics} to the space of linear syzygies of highest order.\smallskip 

It would appear to us to be very difficult to adapt degeneration methods to prove the structure Theorem \ref{geo-thm-intro}, as opposed to merely establishing the vanishing from Voisin's original result. For instance, the construction of Lazarsfeld--Mukai bundles fails on the (non-normal) tangent developable, so that it is not even clear how the bundle $E$ degenerates to this surface. \smallskip

There are no known conjectural candidates for an analogous result to Theorem \ref{geo-thm-intro} in odd genus $g=2k+1$. In this case, the dimension of $\mathrm{K}_{k-1,1}(X,L)$ is not given by a binomial coefficient, so this space cannot be of the form $\mathrm{Sym}^p(V)$ for any vector space $V$. Furthermore, we no longer have uniqueness of the relevant Lazarsfeld-Mukai bundle in this situation.\smallskip

Our argument is formal, using general results on vector bundles rather than a detailed study of the geometry of curves. Consequently, we expect our approach to generalize well to higher dimensional varieties, for which previous approaches to Green's conjecture seem less applicable. See \cite{ein-laz-arb-dim} for applications of vector bundle methods to syzygies of varieties of high dimension.\smallskip

\textbf{Acknowledgements} I thank C.\ Voisin for helpful explanations and G.\ Farkas for numerous discussions. I thank R.\ Lazarsfeld for encouragement and for detailed comments. I thank M.\ Aprodu, J.\ Ellenberg, D.\ Erman, D.\ Huybrechts, J.\ Rathmann, E.\ Sernesi and R.\ Yang for feedback on previous versions.\smallskip

I thank the referee for a careful reading and for comments which greatly improved the exposition. The author is supported by NSF grant DMS-1701245.

\subsection{Preliminaries}
We work over the complex numbers throughout. We gather here a few facts. We have a natural isomorphism $(\mathrm{Sym}^j F)^{\vee} \simeq \mathrm{Sym}^j F^{\vee}$ for any vector bundle $F$ on a variety $X$ defined over $\C$ (over arbitrary fields, this statement requires the characteristic to be at least $j+1$). Let $0 \to F_1 \to F_2 \to F_3 \to 0$ be an exact sequence of bundles on $X$. If $p=0$ or $p \geq i+1$ then, 
from \cite[\S V]{ABW}, we have exact sequences
{\small{\begin{align*}
&\ldots \to \bigwedge^{i-2} F_2 \otimes \mathrm{Sym}^2(F_1) \to \bigwedge^{i-1} F_2 \otimes F_1  \to \bigwedge^i F_2 \to \bigwedge^i F_3 \to 0, \\
& \ldots \to \mathrm{Sym}^{i-2}F_2 \otimes \bigwedge^2 F_1 \to \mathrm{Sym}^{i-1}F_2 \otimes F_1 \to \mathrm{Sym}^{i}F_2 \to \mathrm{Sym}^{i}F_3 \to 0. 
\end{align*}}}
 We may further dualize $0 \to F_1 \to F_2 \to F_3 \to 0$, form the second exact sequence above and dualize again to obtain an exact sequence
{\small{$$
		0 \to \mathrm{Sym}^i(F_1) \to \mathrm{Sym}^{i}(F_2) \to \mathrm{Sym}^{i-1}(F_2) \otimes F_3 \to \mathrm{Sym}^{i-2}(F_2) \otimes \bigwedge^2 F_3 \to \ldots
$$}}
Let $f: X \to Y$ be a morphism of varieties and $\mathcal{F} \in \mathrm{Coh}(X)$ a sheaf. If $\mathcal{E}$ is a vector bundle on $Y$ then we have the \emph{Projection Formula} $\mathrm{R}^if_*(\mathcal{F} \otimes f^* \mathcal{E}) \simeq \mathrm{R}^if_*\mathcal{F} \otimes \mathcal{E}$, \cite[III, Ex.\ 8.3]{hartshorne}. In particular, $f_*(\mathcal{F} \otimes f^* \mathcal{E}) \simeq f_*\mathcal{F} \otimes \mathcal{E}$. If $\mathrm{R}^if_*\mathcal{F}=0$ for all $ i>0$ then $\mathrm{H}^p(X,\mathcal{F}) \simeq \mathrm{H}^p(Y, f_* \mathcal{F})$ for $p \geq 0$, \cite[III, Ex.\ 8.1 ]{hartshorne}. If $X, Y$ are varieties and $\mathcal{F} \in \mathrm{Coh}(X), \mathcal{G} \in \mathrm{Coh}(Y)$ are sheaves, the \emph{K\"unneth formula} states
 {\small{$$\mathrm{H}^{m}(X \times Y, \mathcal{F} \boxtimes \mathcal{G}) \simeq \bigoplus_{a+b=m} \mathrm{H}^a(X,\mathcal{F}) \otimes \mathrm{H}^b(Y, \mathcal{G}),$$}}
where $\mathcal{F} \boxtimes \mathcal{G}:=p^*\mathcal{F} \otimes q^* \mathcal{G}$, for projections $p: X \times Y \to X$, $q: X \times Y \to Y$.\smallskip
 
 Assume we have an exact sequence $0 \to \mathcal{F} \to \mathcal{G} \to \mathcal{H} \to 0$ of coherent sheaves on a quasi-projective variety, with $\mathcal{G}$ locally free. Assume either $\mathcal{H}$ is locally free or $\mathcal{H} \simeq \mathcal{O}_D$ for a Cartier divisor $D$. Then $\mathcal{F}$ is locally free. This follows from \cite[III, Ex 6.5]{hartshorne}.

\section{Voisin's Theorem in Even Genus} \label{sec1}
Let $X$ be a K3 surface of Picard rank one and even genus $g=2k$. Let $M_L$ denote the bundle defined by the sequence $0 \to M_L \to \mathrm{H}^0(L)\otimes \mathcal{O}_X \to L \to 0.$ 
By the Kernel Bundle description of Koszul cohomology, see \cite{lazarsfeld-VBT} or \cite[\S 3]{ein-lazarsfeld-asymptotic}, to prove Voisin's Theorem it suffices to show $$\mathrm{H}^1(X, \bigwedge^{k+1} M_L )=0.$$
Consider the unique rank two, Lazarsfeld--Mukai, bundle $E$ on $X$ as in the introduction. For general $s \in \mathrm{H}^0(E)$, the zero-locus $Z(s)$ corresponds to a $g^1_{k+1}$ on a smooth $C \in |L|$. 
For \emph{any} $s \in \mathrm{H}^0(E)$, $Z(s) \seq X$ is zero-dimensional and we have an exact sequence $$0 \to \mathcal{O}_X \xrightarrow{s} E \xrightarrow{\wedge s} I_{Z(s)} \otimes L \to 0.$$

Our starting point will be the \emph{Secant Sheaf} approach to the study of syzygies, which was originally utilized by Green and Lazarsfeld to give a simple proof of Petri's Theorem \cite{green-laz-petri}. For any $t \in \mathrm{H}^0(E)$, define the vector space $W_t = \mathrm{H}^0(L)/ \mathrm{H}^0(L \otimes I_{Z(t)}).$
We define the \emph{secant sheaves} associated to the secant locus $Z(t)$ by
{\small{\begin{align*}
\Gamma'_t &:=\mathrm{Ker}\left(W_t \otimes \mathcal{O}_X \to L_{|_{Z(t)}} \right),\\
\mathcal{S}'_t &:=\mathrm{Ker}\left( \mathrm{H}^0(L \otimes I_{Z(t)}) \otimes \mathcal{O}_X \to L \otimes I_{Z(t)} \right).
\end{align*}}}
Since $L \otimes I_{Z(t)}$ is a quotient of the globally generated vector bundle $E$, the map $\mathrm{H}^0(L \otimes I_{Z(t)}) \otimes \mathcal{O}_X \to L \otimes I_{Z(t)}$ is surjective. These sheaves are related to $M_L$ via the commutative diagram
{\small{$$\begin{tikzcd}
& 0 \arrow[d] & 0 \arrow[d]  & 0 \arrow[d] \\
0 \arrow[r] & \mathcal{S}'_t \arrow[r] \arrow[d] & \mathrm{H}^0(L \otimes I_{Z(t)}) \otimes \mathcal{O}_X  \arrow[r] \arrow[d] &  L \otimes I_{Z(t)} \arrow[r] \arrow[d] &0\\
0 \arrow[r] & M_L \arrow[r] \arrow[d] & \mathrm{H}^0(L) \otimes \mathcal{O}_X  \arrow[r] \arrow[d]  &  L \arrow[r] \arrow[d]& 0\\
0 \arrow[r] & \Gamma'_t \arrow[r] \arrow[d] & W_t \otimes \mathcal{O}_X  \arrow[r] \arrow[d]  &  L_{|_{Z(t)}}  \arrow[r] \arrow[d] & 0\\
& 0 & 0 & 0 &
\end{tikzcd},$$}}
with exact rows and columns. By considering the first column, 
$$0 \to \mathcal{S}'_t  \to M_L \to \Gamma'_t \to 0$$
one can relate syzygies of $X$ to the cohomological properties of the sheaves $\Gamma'_t$ and $\mathcal{S}'_t$. This is a special case of the approach taken in the influential paper \cite{ein-lazarsfeld-asymptotic}. Two problems, however, present themselves. Firstly, since $Z(t)$ is not a divisor, the sheaf $\Gamma'_t$ is not locally free. To resolve this, let $\pi_t : B_t \to X$ be the blow-up in the local complete intersection $Z(t)$, with exceptional divisor $D_t$. Identify $W_t$ with $\mathrm{H}^0(B_t,\pi_t^*L)/ \mathrm{H}^0(B_t, \pi_t^*L\otimes I_{D_t})$ and define \emph{vector bundles}
{\small{\begin{align*}
\Gamma_t &:=\mathrm{Ker}\left(W_t \otimes \mathcal{O}_{B_t} \to \pi_t^*L_{|_{D_t}} \right),\\
\mathcal{S}_t &:=\mathrm{Ker}\left( \mathrm{H}^0(\pi_t^*L \otimes I_{D_t}) \otimes \mathcal{O}_B \to L \otimes I_{D_t} \right).
\end{align*}}}
One may now try to work with the sequence $0 \to \mathcal{S}_t  \to \pi_t^*M_L \to \Gamma_t \to 0$. However, each section $t \in \mathrm{H}^0(E)$ comes on an equal footing. It is therefore more natural to work with all of the sections \emph{simultaneously}.\smallskip

We now adapt the above construction with these considerations in mind. Set $\PP:=\PP(\mathrm{H}^0(E))\simeq \PP^{k+1}$. Consider $X \times \PP$ with projections $p: X \times \PP \to X$, $q: X \times \PP \to \PP$. Define $\mathcal{Z} \seq X \times \PP$ as the locus $\left\{ (x,s) \; | \; s(x)=0 \right\}$. Since $E$ is globally generated, $\mathcal{Z}$ is a projective bundle over $X$ and hence smooth. We have an exact sequence
$$0 \to \mathcal{O}_X \boxtimes \mathcal{O}_{\PP}(-2) \xrightarrow{\mathrm{id}} E \boxtimes \mathcal{O}_{\PP}(-1) \to p^*L \otimes I_{\mathcal{Z}} \to 0,$$
where the first nonzero map is given by multiplication by $$\mathrm{id} \in \mathrm{H}^0\left(E \boxtimes \mathcal{O}_{\PP}(1)\right) \simeq \mathrm{H}^0(E) \otimes \mathrm{H}^0(E)^{\vee} \simeq \mathrm{Hom}\left(\mathrm{H}^0(E), \mathrm{H}^0(E)\right).$$ 
Note that $\mathcal{Z} \to \PP$ is finite and flat, by the ``miracle flatness" theorem \cite[Prop.\ 6.1.5]{EGA}.

\begin{rem}
As soon as there exists a nontrivial, effective divisor $C$ on $X$ with $\mathrm{H}^0(E(-C))$ nonzero, then $\mathcal{Z} \to \PP$ cannot be finite and flat. For this reason, it is essential that $\mathrm{Pic}(X) \simeq \mathbb{Z}[L]$.
\end{rem}
\smallskip

Let $\mathcal{M}:=p^* M_L$. Note $\mathrm{H}^1(X \times \PP,\bigwedge^{k+1}\mathcal{M}) \simeq \mathrm{H}^1(X, \bigwedge^{k+1}M_L) $ by the K\"unneth formula, as $\mathrm{H}^1(\mathcal{O}_X)=0$. Let $\pi:B \to  X \times \PP $ be the blow-up along $\mathcal{Z}$ with exceptional divisor $D$. Then 
$\pi_*\mathcal{O}_B \simeq \mathcal{O}_{X \times \PP}, \; \; \pi_* I_{D} \simeq I_{\mathcal{Z}} \; \; \text{and} \; \; \mathrm{R}^i \pi_*\mathcal{O}_B=\mathrm{R}^i\pi_* I_D=0 \; \text{for $i>0$},$
cf.\ \cite[V, Prop.\ 3.4 and Ex.\ 3.1]{hartshorne}. Set $p':=p \circ \pi$, $q':=q \circ \pi$. We have canonical identifications
$$q'_*({p'}^*L \otimes I_D) \simeq q_*(p^*L \otimes I_{\mathcal{Z}}), \; \; \; \; \; \; \; q'_*{p'}^*L  \simeq q_* p^*L.$$
Consider $\mathcal{W}:=\text{Coker}\left(q'_*({p'}^*L \otimes I_D) \to q'_*{p'}^*L\right)\simeq  \text{Coker}\left(q_*({p}^*L \otimes I_{\mathcal{Z}}) \to q_*{p}^*L\right).$
\begin{lem} \label{very-first-lem}
The sheaf $\mathcal{W}$ is locally free of rank $k$.
\end{lem}
\begin{proof}
Applying $\mathrm{R}q_*$ gives the exact sequence $0 \to \mathcal{W} \to q_*(p^*L_{|_{\mathcal{Z}}}) \to \mathrm{R}^1q_*(p^*L \otimes I_{\mathcal{Z}})\to 0$. For any $s \in \mathrm{H}^0(E)$, we have $\mathrm{H}^1(X,L\otimes I_{Z(s)}) \simeq \mathrm{H}^2(\mathcal{O}_X)$. Thus $q_*(p^*L_{|_{\mathcal{Z}}}) $ and $\mathrm{R}^1q_*(p^*L \otimes I_{\mathcal{Z}})$ are locally free of ranks $k+1$ and $1$, by Grauert's Theorem \cite[III, \S 12]{hartshorne}. The claim follows.
\end{proof}
\smallskip

Since $D$ is a divisor, we have a rank $k$ vector bundle 
$\displaystyle{\Gamma:=\text{Ker}\left({q'}^*\mathcal{W} \twoheadrightarrow {p'}^*L_{|_D} \right).}$
As $L \otimes I_{Z(s)}$ is globally generated for all $s \in \mathrm{H}^0(E)$, we have a natural surjection ${q}^* q_* ({p}^*L \otimes I_{\mathcal{Z}}) \to {p}^*L \otimes I_{\mathcal{Z}}$. Applying $\pi^*$ and noting that $I_D$ is a quotient of $\pi^*I_{\mathcal{Z}}$, we have a surjection ${q'}^* q'_* ({p'}^*L \otimes I_D) \to {p'}^*L \otimes I_D$.
Let $\mathcal{S}$ be the vector bundle on $B$ defined by the exact sequence
$$0 \to \mathcal{S} \to {q'}^* q'_* ({p'}^*L \otimes I_D) \to {p'}^*L \otimes I_D \to 0.$$
We call $\mathcal{S}$ and $\Gamma$ the \emph{universal secant bundles}.
We have an exact sequence $$0 \to \mathcal{S} \to \pi^* \mathcal{M} \to \Gamma \to 0$$ which gives the exact sequence
$$ \ldots \to \bigwedge^{k-1} \pi^* \mathcal{M} \otimes \mathrm{Sym}^2 \mathcal{S} \to  \bigwedge^{k} \pi^* \mathcal{M} \otimes \mathcal{S} \to \bigwedge^{k+1}\pi^* \mathcal{M} \to 0.$$
To prove Voisin's Theorem it suffices to show $$\mathrm{H}^i(B, \bigwedge^{k+1-i} \pi^* \mathcal{M} \otimes \mathrm{Sym}^i \mathcal{S})=0 \; \; \text{for $1 \leq i \leq k+1$}.$$
One readily shows these vanishings for $i <k+1$ (see Theorem \ref{main-thm}). The crucial point is to show $\mathrm{H}^{k+1}(B, \mathrm{Sym}^{k+1} \mathcal{S})=0$. To ease the notation, we set $$\mathcal{G}:=q^*q_*(p^*L \otimes I_{\mathcal{Z}}).$$ Note also that $q_*p^*E \simeq \mathrm{H}^0(E) \otimes \mathcal{O}_{\PP}$.
\begin{lem} \label{first-iso}
We have $\mathrm{H}^{k+1}(X \times \PP,\mathrm{Sym}^{k+1} \mathcal{G})=0$ as well as a natural isomorphism
$$\mathrm{H}^k(X \times \PP,\mathrm{Sym}^{k+1} \mathcal{G})\simeq \mathrm{Sym}^k \mathrm{H}^0(E).$$
\end{lem}
\begin{proof}
The sequence $0 \to q^*\mathcal{O}_{\PP}(-2) \to q^*q_*p^*E \otimes q^*\mathcal{O}_{\PP}(-1) \to \mathcal{G}\to 0$
 gives the exact sequence
$$0 \to \mathrm{Sym}^k \mathrm{H}^0(E) \otimes q^*\mathcal{O}_{\PP}(-k-2) \to \mathrm{Sym}^{k+1}\mathrm{H}^0(E) \otimes q^*\mathcal{O}_{\PP}(-k-1) \to \mathrm{Sym}^{k+1}\mathcal{G} \to 0,$$
since $q_*p^*E \simeq \mathrm{H}^0(E) \otimes \mathcal{O}_{\PP}$. Thus
\begin{align*}
\mathrm{H}^k(\mathrm{Sym}^{k+1}\mathcal{G}) \simeq \mathrm{H}^{k+1}(\mathrm{Sym}^k \mathrm{H}^0(E) \otimes q^*\mathcal{O}_{\PP}(-k-2))\simeq \mathrm{Sym}^k \mathrm{H}^0(E)
\end{align*}
The vanishing $\mathrm{H}^{k+1}(\mathrm{Sym}^{k+1} \mathcal{G})=0$ follows from $$\mathrm{H}^{k+1}(\mathcal{O}_X \boxtimes \mathcal{O}_{\PP}(-k-1))=\mathrm{H}^{k+2}(\mathcal{O}_X \boxtimes \mathcal{O}_{\PP}(-k-2))=0,$$
using $\mathrm{H}^1(\mathcal{O}_X)=0$.
\end{proof}

The next lemma is a similar computation to the previous one.
\begin{lem} \label{first-twist-lem}
We have a natural isomorphism
$\mathrm{H}^k(\mathrm{Sym}^{k} \mathcal{G} \otimes p^*L \otimes I_{\mathcal{Z}}) \simeq \mathrm{Sym}^k \mathrm{H}^0(E).$

\end{lem}
\begin{proof}
We have the short exact sequence
\begin{align*}
0 &\to \mathrm{Sym}^{k-1} \mathrm{H}^0(E) \otimes q^*\mathcal{O}_{\PP}(-k-1) \otimes p^*L \otimes I_{\mathcal{Z}} \to \mathrm{Sym}^{k}\mathrm{H}^0(E) \otimes q^*\mathcal{O}_{\PP}(-k) \otimes p^*L \otimes I_{\mathcal{Z}} \\
&\to \mathrm{Sym}^{k}\mathcal{G} \otimes p^*L \otimes I_{\mathcal{Z}}\to 0,
\end{align*}
as well as the exact sequence
$0 \to \mathcal{O}_X \boxtimes \mathcal{O}_{\PP}(-2) \to E \boxtimes \mathcal{O}_{\PP}(-1) \to p^*L \otimes I_{\mathcal{Z}} \to 0.$\\


By the K\"unneth formula, $\mathrm{H}^{k+2}(\mathcal{O}_X \boxtimes \mathcal{O}_{\PP}(-k-3))=0$. We have
$$\mathrm{H}^{k+1}(\mathcal{O}_X \boxtimes \mathcal{O}_{\PP}(-k-3))=\mathrm{H}^{k+1}(\mathrm{K}_{\PP}(-1)) \simeq \mathrm{H}^0(\mathcal{O}_{\PP}(1))^{\vee}\simeq \mathrm{H}^0(E).$$
Further, $\mathrm{H}^{k+1}(E \boxtimes \mathcal{O}_{\PP}(-k-2)) \simeq \mathrm{H}^{k+1}(E \boxtimes \mathrm{K}_{\PP}) \simeq \mathrm{H}^0(E).$
The map $$\mathrm{H}^{k+1}(\mathcal{O}_X \boxtimes \mathcal{O}_{\PP}(-k-3)) \to \mathrm{H}^{k+1}(E \boxtimes \mathcal{O}_{\PP}(-k-2)) $$ is identified with $\mathrm{id}: \mathrm{H}^0(E) \to \mathrm{H}^0(E)$. Thus $\mathrm{H}^{k+1}(L \boxtimes \mathcal{O}_{\PP}(-k-1)\otimes I_{\mathcal{Z}})=0.$ We likewise have $\mathrm{H}^{k}(L \boxtimes \mathcal{O}_{\PP}(-k-1)\otimes I_{\mathcal{Z}})=0.$ 
Thus,
\begin{align*}
\mathrm{H}^k(\mathrm{Sym}^{k} \mathcal{G} \otimes p^*L \otimes I_{\mathcal{Z}}) 
\simeq \mathrm{Sym}^k\mathrm{H}^0(E)\otimes \mathrm{H}^k(L \boxtimes \mathcal{O}_{\PP}(-k)\otimes I_{\mathcal{Z}}).
\end{align*}
To finish the proof, it suffices to show that the boundary map
$$\mathrm{H}^k(L \boxtimes \mathcal{O}_{\PP}(-k)\otimes I_{\mathcal{Z}}) \to \mathrm{H}^{k+1}(q^*\mathrm{K}_{\PP})$$
is an isomorphism, which follows from the fact that $\mathrm{H}^i(E \boxtimes \mathcal{O}(-k-1))=0$ for all $i$.
\end{proof}

We now repeat the previous lemma, twisting instead by $\mathcal{G}:=q^*q_*(p^*L \otimes I_{\mathcal{Z}})$.
\begin{lem} \label{second-twist-lem}
The evaluation morphism $\mathcal{G} \twoheadrightarrow p^*L \otimes I_{\mathcal{Z}}$ induces an isomorphism
$$\mathrm{H}^k(\mathrm{Sym}^{k} \mathcal{G} \otimes \mathcal{G}) \xrightarrow{\sim} \mathrm{H}^k(\mathrm{Sym}^{k} \mathcal{G} \otimes p^*L \otimes I_{\mathcal{Z}}) .$$
\end{lem}
\begin{proof}
We have the short exact sequences
\begin{align*}
0 \to \mathrm{Sym}^{k-1} \mathrm{H}^0(E) \otimes q^*\mathcal{O}_{\PP}(-k-1) \otimes \mathcal{G} \to \mathrm{Sym}^{k}\mathrm{H}^0(E) \otimes q^*\mathcal{O}_{\PP}(-k) \otimes \mathcal{G} \to \mathrm{Sym}^{k}\mathcal{G} \otimes \mathcal{G} \to 0,
\end{align*}
and $0 \to q^*\mathcal{O}_{\PP}(-2) \to \mathrm{H}^0(E) \otimes q^*\mathcal{O}_{\PP}(-1) \to \mathcal{G} \to 0.$\smallskip

Using the second sequence, $\mathrm{H}^{k+1}(\mathcal{G} \otimes q^*\mathcal{O}_{\PP}(-k-1))=\mathrm{H}^{k}(\mathcal{G} \otimes q^*\mathcal{O}_{\PP}(-k-1))=0$ and $\mathrm{H}^k(\mathcal{G} \otimes q^*\mathcal{O}_{\PP}(-k))\xrightarrow{\sim} \mathrm{H}^{k+1}(q^*\mathrm{K}_{\PP})$. Thus
\begin{align*}
\mathrm{H}^k(\mathrm{Sym}^{k} \mathcal{G} \otimes \mathcal{G}) \simeq \mathrm{H}^k(\mathrm{Sym}^k \mathrm{H}^0(E) \otimes q^*\mathcal{O}_{\PP}(-k)\otimes \mathcal{G}) \simeq \mathrm{Sym}^k\mathrm{H}^0(E)
\end{align*}
and the evaluation map gives an isomorphism $\mathrm{H}^k(\mathrm{Sym}^{k}\mathcal{G} \otimes \mathcal{G}) \xrightarrow{\sim} \mathrm{H}^k(\mathrm{Sym}^{k} \mathcal{G} \otimes p^*L \otimes I_{\mathcal{Z}}).$
\end{proof}
As a corollary, we now deduce:
\begin{prop} \label{auto-injectivity}
We have a natural isomorphism $\mathrm{H}^k(\mathrm{Sym}^{k+1} \mathcal{G}) \xrightarrow{\sim} \mathrm{H}^k(\mathrm{Sym}^{k} \mathcal{G} \otimes p^*L \otimes I_{\mathcal{Z}}).$
\end{prop}
\begin{proof}
By the previous lemmas, it suffices to show that the natural morphism
$$\mathrm{H}^k(\mathrm{Sym}^{k+1} \mathcal{G}) \to \mathrm{H}^k(\mathrm{Sym}^{k} \mathcal{G} \otimes \mathcal{G})$$
is injective. For a vector bundle $\mathcal{F}$ the composition $\mathrm{Sym}^i \mathcal{F} \to \mathrm{Sym}^{i-1}\mathcal{F} \otimes \mathcal{F} \to \mathrm{Sym}^i \mathcal{F}$
of natural maps is just multiplication by $i$. This completes the proof.
\end{proof}
We now complete the proof that $\mathrm{K}_{k,1}(X,L)=0$.
\begin{thm} \label{main-thm}
We have $\mathrm{H}^i(B, \bigwedge^{k+1-i} \pi^* \mathcal{M} \otimes \mathrm{Sym}^i \mathcal{S})=0$ for $1 \leq i \leq k+1$.
\end{thm}
\begin{proof}
Observe $\pi^*\mathcal{G}\simeq {q'}^*q'_*({p'}^*L \otimes I_{D})$. From the defining sequence for $\mathcal{S}$, we have an exact sequence
\begin{align} \label{ses-S}
0 \to \mathrm{Sym}^{i}\mathcal{S} \to \mathrm{Sym}^{i} \pi^* \mathcal{G} \to \mathrm{Sym}^{i-1} \pi^*\mathcal{G} \otimes {p'}^*L \otimes I_{D} \to 0.
\end{align}
Using the projection formula, and recalling the identities $\pi_*\mathcal{O}_B \simeq \mathcal{O}_{X \times \PP}$, $\pi_* I_{D} \simeq I_{\mathcal{Z}}$ and $\mathrm{R}^j \pi_*\mathcal{O}_B=\mathrm{R}^j\pi_* I_D=0$ for $j>0$,  we may identify
$$\mathrm{H}^{\ell}(B, \mathrm{Sym}^{i} \pi^* \mathcal{G}) \to \mathrm{H}^{\ell}(B,\mathrm{Sym}^{i-1} \pi^*\mathcal{G} \otimes {p'}^*L \otimes I_{D})$$
with the natural map $\mathrm{H}^{\ell}(X \times \PP, \mathrm{Sym}^{i} \mathcal{G}) \to \mathrm{H}^{\ell}(X \times \PP,\mathrm{Sym}^{i-1} \mathcal{G} \otimes {p}^*L \otimes I_{\mathcal{Z}})$, for any $\ell$. Taking the long exact sequence of cohomology for the sequence (\ref{ses-S}) for $i=k+1$ and applying the previous lemmas, we immediately see $\displaystyle \mathrm{H}^{k+1}(B,\mathrm{Sym}^{k+1} \mathcal{S})=0.$ Namely, $\mathrm{H}^{k+1}(\mathrm{Sym}^{k+1}\pi^* \mathcal{G})=0$ from Lemma \ref{first-iso} and the map $\mathrm{H}^k(\mathrm{Sym}^{k+1} \pi^* \mathcal{G}) \to \mathrm{H}^k(\mathrm{Sym}^{k}\pi^* \mathcal{G}  \otimes {p'}^*L \otimes I_{D} )$ is an isomorphism by Proposition \ref{auto-injectivity}.
\smallskip

To complete the proof, it suffices to show
\begin{align*}
\mathrm{H}^i (\bigwedge^{k+1-i} \pi^* \mathcal{M} \otimes \mathrm{Sym}^{i} \pi^* \mathcal{G})=\mathrm{H}^{i-1}(\bigwedge^{k+1-i} \pi^* \mathcal{M} \otimes \mathrm{Sym}^{i-1} \pi^* \mathcal{G} \otimes {p'}^*L \otimes I_{D} )=0, \; \; \text{for $1 \leq i \leq k$}.
\end{align*}
The first vanishing follows from the exact sequence 
$$0 \to \mathrm{Sym}^{i-1} \mathrm{H}^0(E) \otimes q^*\mathcal{O}_{\PP}(-i-1) \to \mathrm{Sym}^{i}\mathrm{H}^0(E) \otimes q^*\mathcal{O}_{\PP}(-i) \to \mathrm{Sym}^{i}\mathcal{G}\to 0,$$
together with $\mathrm{H}^i(\bigwedge^{k+1-i} M_L \boxtimes \mathcal{O}_{\PP}(-i))=\mathrm{H}^{i+1}(\bigwedge^{k+1-i}M_L \boxtimes \mathcal{O}_{\PP}(-i-1))=0$ for $1 \leq i \leq k$.\smallskip

Next, from the exact sequence
\begin{align*}
0 &\to \mathrm{Sym}^{i-2} \mathrm{H}^0(E) \otimes q^*\mathcal{O}_{\PP}(-i) \otimes p^*L \otimes I_{\mathcal{Z}} \to \mathrm{Sym}^{i-1}\mathrm{H}^0(E) \otimes q^*\mathcal{O}_{\PP}(-i+1) \otimes p^*L \otimes I_{\mathcal{Z}}\\ &\to \mathrm{Sym}^{i-1}\mathcal{G} \otimes p^*L \otimes I_{\mathcal{Z}}\to 0,
\end{align*}
it suffices to show $\mathrm{H}^{i-1}(\bigwedge^{k+1-i} M_L (L)\boxtimes \mathcal{O}_{\PP}(-i+1)\otimes I_{\mathcal{Z}})=\mathrm{H}^{i}(\bigwedge^{k+1-i} M_L(L) \boxtimes \mathcal{O}_{\PP}(-i) \otimes I_{\mathcal{Z}})=0$ for $1 \leq i \leq k.$ This follows from $0 \to \mathcal{O}_X \boxtimes \mathcal{O}_{\PP}(-2) \to E \boxtimes \mathcal{O}_{\PP}(-1) \to L \otimes I_{\mathcal{Z}} \to 0,$ after twisting with $\bigwedge^{k+1-i} M_L \boxtimes \mathcal{O}_{\PP}(q)$ for $q \in \{-i+1,-i \}$, as $\mathrm{H}^0(X,M_L)=0$ if $i=k$. \end{proof}

\section{The Geometric Syzygy Conjecture} \label{sec2}
In this section, we use the techniques used in our proof of Voisin's Theorem to resolve the Geometric Syzygy Conjecture for extremal syzygies of generic curves of even genus. We stick with the notation from Section \ref{sec1}.  As before, consider the exact sequence $0 \to \mathcal{S} \to \pi^* \mathcal{M} \to \Gamma \to 0.$
\begin{thm} \label{real-thm}
The natural morphism $\psi: \mathrm{H}^1(B, \wedge^k \pi^* \mathcal{M} \otimes {p'}^*L) \to \mathrm{H}^1(B, \wedge^k \Gamma \otimes {p'}^*L)$ is surjective.
\end{thm}
\begin{proof}
From the exact sequence
\begin{align*}
\ldots \to \mathcal{S} \otimes \wedge^{k-1} \pi^* \mathcal{M}   \otimes {p'}^*L \to \wedge^k  \mathcal{M}   \otimes {p'}^*L \to \wedge^k \Gamma \otimes {p'}^*L \to 0,
\end{align*}
it suffices to show $\mathrm{H}^{1+i}(\mathrm{Sym}^i \mathcal{S} \otimes \wedge^{k-i} \pi^* \mathcal{M} \otimes {p'}^*L)=0$
for $1 \leq i \leq k$. From the exact sequence
$$0 \to \mathrm{Sym}^{i} \mathcal{S} \to \mathrm{Sym}^i( \pi^*\mathcal{G}) \to \mathrm{Sym}^{i-1}(\pi^*\mathcal{G})\otimes {p'}^*L \otimes I_D \to 0$$
it suffices to have
\begin{align*}
\mathrm{H}^{1+i}\left(X \times \PP,\mathrm{Sym}^i (\mathcal{G}) \otimes \wedge^{k-i} \mathcal{M} \otimes {p}^*L\right)=\mathrm{H}^i\left(X \times \PP,\mathrm{Sym}^{i-1}(\mathcal{G})\otimes \wedge^{k-i}  \mathcal{M} \otimes {p}^*L^{\otimes 2} \otimes I_{\mathcal{Z}}\right)=0
\end{align*}
By taking $\mathrm{Sym}^i$ of the exact sequence
$$0 \to q^*\mathcal{O}_{\PP}(-2) \to \mathrm{H}^0(E) \otimes q^*\mathcal{O}_{\PP}(-1) \to \mathcal{G} \to 0$$
it suffices to show the four vanishings
{\small{\begin{align*}
&\mathrm{H}^{2+i}\left( \wedge^{k-i} M_L (L) \boxtimes   \mathcal{O}_{\PP}(-i-1)  \right)=0, & &\mathrm{H}^{1+i}\left(\wedge^{k-i} M_L (L) \boxtimes \mathcal{O}_{\PP}(-i)  \right)=0,\\
&\mathrm{H}^{1+i}\left((\wedge^{k-i} M_L (2L) \boxtimes \mathcal{O}_{\PP}(-i) )\otimes I_{\mathcal{Z}} \right)=0, & &\mathrm{H}^{i}\left((\wedge^{k-i} M_L (2L) \boxtimes \mathcal{O}_{\PP}(-i+1) )\otimes I_{\mathcal{Z}} \right)=0.
\end{align*}}}
The first two claims follow from the K\"unneth formula, for $1 \leq i \leq k$. For the last two claims, we use the short exact sequence
$$0 \to L^{-1} \boxtimes \mathcal{O}_{\PP}(-2) \to E(L^{-1}) \boxtimes \mathcal{O}_{\PP}(-1) \to I_{\mathcal{Z}} \to 0,$$
so it suffices to have the four vanishings
{\small{\begin{align*}
&\mathrm{H}^{2+i}\left(\wedge^{k-i} M_L (L) \boxtimes \mathcal{O}_{\PP}(-i-2)  \right)=0,
& &\mathrm{H}^{1+i}\left( \wedge^{k-i} M_L \otimes E(L) \boxtimes \mathcal{O}_{\PP}(-i-1)  \right)=0,\\
&\mathrm{H}^{1+i}\left( \wedge^{k-i} M_L (L) \boxtimes \mathcal{O}_{\PP}(-i-1)  \right)=0,
& &\mathrm{H}^{i}\left(\wedge^{k-i} M_L \otimes E(L) \boxtimes \mathcal{O}_{\PP}(-i)  \right)=0.
\end{align*}}}
This follows from the K\"unneth formula, using $\mathrm{H}^1(X,L)=0$ if $i=k$ in the first vanishing.
\end{proof}

The Geometric Syzygy Conjecture in even genus now follows readily from Theorem \ref{real-thm}. 
\begin{lem}
With notation as in Section \ref{sec1}, we have $\wedge^k \Gamma \simeq I_D \otimes {q'}^* \mathcal{O}_{\PP}(k)$.
\end{lem}
\begin{proof}
By taking determinants of the exact sequence $0 \to \Gamma \to {q'}^* \mathcal{W} \to {p'}^* L_{|_D} \to 0$, it suffices to show $\det \mathcal{W} \simeq \mathcal{O}_{\PP}(k)$. From $0 \to q_*(p^*L \otimes I_{\mathcal{Z}}) \to q_*p^*L \to \mathcal{W} \to 0$, we see $\det \mathcal{W}=\det (q_*(p^*L \otimes I_{\mathcal{Z}}))^{-1}$. We now deduce the claim from the exact sequence
$$0 \to \mathcal{O}_{\PP}(-2) \to \mathrm{H}^0(E) \otimes \mathcal{O}_{\PP}(-1) \to q_*(p^*L \otimes I_{\mathcal{Z}}) \to 0.$$\end{proof}
\begin{cor} \label{main-cor-nat-iso}
The morphism $\psi: \mathrm{H}^1(\wedge^k \pi^* \mathcal{M} \otimes {p'}^*L) \to \mathrm{H}^1(\wedge^k \Gamma \otimes {p'}^*L)$ is an isomorphism and induces a natural isomorphism $\mathrm{K}_{k-1,1}(X,L) \simeq \mathrm{Sym}^{k-2}\mathrm{H}^0(X,E)  $.
\end{cor}
\begin{proof}
We have $\mathrm{H}^1(B, I_D \otimes {q'}^* \mathcal{O}_{\PP}(k))\simeq \mathrm{H}^1(X \times \PP, I_{\mathcal{Z}} \otimes {q}^* \mathcal{O}_{\PP}(k))$. From
$$0 \to \mathcal{O}_X \boxtimes \mathcal{O}_{\PP}(-2) \to E \boxtimes \mathcal{O}_{\PP}(-1) \to p^*L \otimes  I_{\mathcal{Z}} \to 0,$$
we obtain an isomorphism $\mathrm{H}^1(B, \wedge^k \Gamma \otimes {p'}^*L) \simeq \mathrm{H}^0(\mathcal{O}_{\PP}(k-2)) \simeq {\mathrm{Sym}^{k-2}\mathrm{H}^0(X,E)}^{\vee}$. Theorem \ref{real-thm} gives a surjective map $\mathrm{K}_{k-1,2}(X,L) \to {\mathrm{Sym}^{k-2}\mathrm{H}^0(X,E)}^{\vee}$. Dualizing, we have an injective map $\mathrm{Sym}^{k-2}\mathrm{H}^0(X,E) \to \mathrm{K}_{k-1,1}(X,L).$ Voisin's Theorem in even genus, as proven in Section \ref{sec1}, implies $\dim \mathrm{Sym}^{k-2}\mathrm{H}^0(X,E) = \dim \mathrm{K}_{k-1,1}(X,L)$,  \cite[\S 4.1]{farkas-progress}, completing the proof.\end{proof}

We record an observation for future use. From the sequence $0 \to \mathcal{S} \to \pi^* \mathcal{M} \to \Gamma \to 0$, the isomorphism $\psi: \mathrm{H}^1( \wedge^k \pi^* \mathcal{M} \otimes {p'}^*L) \xrightarrow{\sim} \mathrm{H}^1(\wedge^k \Gamma \otimes {p'}^*L)$ is Serre dual to
$$\psi^{\vee} \; : \; \mathrm{H}^{k+2}( \wedge^k \mathcal{S} \otimes \omega_B) \to \mathrm{H}^{k+2}( \wedge^k \pi^* \mathcal{M}\otimes \omega_B).$$
\begin{proof}[Proof of Theorem \ref{geo-thm-intro}]
We have the exact sequence $0 \to \Gamma' \to q^* \mathcal{W} \to p^*L_{|_{\mathcal{Z}}} \to 0,$ where we have set $\Gamma':=\pi_* \Gamma$. The image of the natural map $\wedge^k \Gamma' \to \wedge^k q^* \mathcal{W}  \simeq \mathcal{O}_{\PP}(k)$ is $I_{\mathcal{Z}} \otimes q^*\mathcal{O}_{\PP}(k)$ from \cite[Cor.\ 3.7]{ein-lazarsfeld-asymptotic}. The induced surjection $\wedge^k \Gamma' \to I_{\mathcal{Z}} \otimes q^*\mathcal{O}_{\PP}(k)$ can be identified with the natural map
$$\wedge^k \pi_* \Gamma \to \pi_*\wedge^k \Gamma \simeq \pi_*(I_D \otimes {q'}^*\mathcal{O}_{\PP}(k)) \hookrightarrow \pi_* \wedge^k {q'}^* \mathcal{W} \simeq \wedge^k q^* \mathcal{W} . $$ The map $\psi: \mathrm{H}^1(B, \wedge^k \pi^* \mathcal{M} \otimes {p'}^*L) \to \mathrm{H}^1(B, \wedge^k \Gamma \otimes {p'}^*L)$ can be identified with the composition
$$\mathrm{H}^1(X\times \PP, \wedge^k\mathcal{M} \otimes p^*L) \to \mathrm{H}^1(X\times \PP, \wedge^k \Gamma' \otimes p^*L) \to \mathrm{H}^1( X\times \PP, (L \boxtimes \mathcal{O}_{\PP}(k))\otimes I_{\mathcal{Z}}),$$
where the first map is induced from the natural surjection $\mathcal{M} \to \Gamma'$. Since $q_*\wedge^k\mathcal{M} \otimes p^*L \simeq \mathrm{H}^0(X, \wedge^k M_L(L)) \otimes \mathcal{O}_{\PP}$, we have $\mathrm{H}^i(q_*\wedge^k\mathcal{M} \otimes p^*L)=0$ for $i>0$. Thus the Leray spectral sequence gives an isomorphism
$$\mathrm{H}^1(X \times \PP, \wedge^k\mathcal{M} \otimes p^*L) \simeq \mathrm{H}^0(\PP, R^1q_*\wedge^k\mathcal{M} \otimes p^*L).$$
Next, the exact sequence $0 \to \mathcal{O}_{\PP}(k-2) \to \mathrm{H}^0(E) \otimes \mathcal{O}_{\PP}(k-1) \to q_*((L \boxtimes \mathcal{O}_{\PP}(k))\otimes I_{\mathcal{Z}}) \to 0$ gives $\mathrm{H}^i(q_*((L \boxtimes \mathcal{O}_{\PP}(k))\otimes I_{\mathcal{Z}}))=0$ for $i >0$ so 
$$\mathrm{H}^1( X \times \PP,(L \boxtimes \mathcal{O}_{\PP}(k))\otimes I_{\mathcal{Z}})\simeq \mathrm{H}^0(\PP, R^1 q_*((L \boxtimes \mathcal{O}_{\PP}(k))\otimes I_{\mathcal{Z}})).$$
The map $\psi$ is thus naturally identified with the global section map applied to the morphism
$$\mathrm{K}_{k-1,1}(X,L)^{\vee} \otimes \mathcal{O}_{\PP} \simeq R^1q_*\wedge^k\mathcal{M} \otimes p^*L \to R^1 q_*((L \boxtimes \mathcal{O}_{\PP}(k))\otimes I_{\mathcal{Z}}),$$
of vector bundles. Note that $R^1 q_*((L \boxtimes \mathcal{O}_{\PP}(k))\otimes I_{\mathcal{Z}})$ is a line bundle from the proof of Lemma \ref{very-first-lem}. From
$0 \to \mathcal{O}_X \boxtimes \mathcal{O}_{\PP}(-2) \to E \boxtimes \mathcal{O}_{\PP}(-1) \to p^*L \otimes  I_{\mathcal{Z}} \to 0,$
we have $$R^1 q_*((L \boxtimes \mathcal{O}_{\PP}(k))\otimes I_{\mathcal{Z}}) \simeq R^2q_*\mathcal{O}_{X \times \PP} \otimes \mathcal{O}_{\PP}(k-2)\simeq \mathcal{O}_{\PP}(k-2),$$
using that $R^2q_*\mathcal{O}_{X \times \PP} \simeq \mathcal{O}_{\PP}$ by relative duality. Since the induced morphism
\begin{align} \label{bundle-ind-vero}
\mathrm{K}_{k-1,1}(X,L)^{\vee} \otimes \mathcal{O}_{\PP} \to R^1 q_*((L \boxtimes \mathcal{O}_{\PP}(k))\otimes I_{\mathcal{Z}}) \simeq \mathcal{O}_{\PP}(k-2)
\end{align}
is surjective on global sections and $\mathcal{O}_{\PP}(k-2)$ is globally generated, it must be surjective as a morphism of sheaves. Morphism (\ref{bundle-ind-vero}) therefore induces a Veronese morphism $$\widetilde{\psi} : \PP(\mathrm{H}^0(E)) \to \PP(\mathrm{K}_{k-1,1}(X,L))$$ of degree $k-2$.\smallskip

To complete the proof, it only remains to show that, for any nonzero $t \in \mathrm{H}^0(E)$, we have
$$\widetilde{\psi}([t])=[\mathrm{K}_{k-1,1}\left(X,L,\mathrm{H}^0(L \otimes I_{Z(t)})\right)].$$
The fiber of the morphism (\ref{bundle-ind-vero}) over $[t]$ is identified with the natural composition
$$\mathrm{H}^1(X,\wedge^k M_L(L)) \to \mathrm{H}^1(X,\wedge^k \Gamma'_t(L)) \to \mathrm{H}^1(X,L \otimes I_{Z(t)}),$$
using the notation from Section \ref{sec1}. As above, this may be identified with the natural surjection
$$\mathrm{H}^1(B_t,\wedge^k \pi^*_tM_L(L)) \twoheadrightarrow \mathrm{H}^1(B_t,\wedge^k \Gamma_t (\pi^*_t L )).$$
From the exact sequence $0 \to \mathcal{S}_t \to \pi_t^* M_L \to \Gamma_t \to 0$ and the isomorphisms $\wedge^k M_L(L) \simeq \wedge^k M_L^{\vee}$, $\wedge^k \Gamma_t (\pi^*_t L) \simeq \wedge^k \mathcal{S}_t^{\vee}$, we may identify the fiber of (\ref{bundle-ind-vero}) with the map $\mathrm{H}^1(X,\wedge^k M_L^{\vee}) \twoheadrightarrow \mathrm{H}^1(B_t, \wedge^k \mathcal{S}_t^{\vee})$, which is Serre dual to
$$ \mathrm{H}^1(B_t, \wedge^k \mathcal{S}_t(D_t)) \hookrightarrow \mathrm{H}^1(B_t,\wedge^k \pi_t^* M_L(D_t)). $$ Using the Kernel Bundle description of syzygies, this may be interpreted as an inclusion
$$\mathrm{K}_{k-1,1}(B_t,D_t, \pi_t^*L-D_t)\hookrightarrow \mathrm{K}_{k-1,1}(B_t,D_t,\pi_t^*L).$$
Since $\mathrm{H}^0(B_t, \pi_t^*L^{\otimes n}(D_t)) \simeq \mathrm{H}^0(B_t, \pi_t^*L^{\otimes n})\simeq \mathrm{H}^0(X, L^{\otimes n})$ for any $n$,
this inclusion is naturally identified with $\mathrm{K}_{k-1,1}\left(X,L,\mathrm{H}^0(L \otimes I_{Z(t)})\right) \hookrightarrow \mathrm{K}_{k-1,1}(X,L),$
as required.
\end{proof}

It is well-known that Theorem \ref{geo-thm-intro}, combined with an observation of Voisin, implies Corollary \ref{geo-cor}, see the unpublished \cite[\S 11]{bothmer-preprint}. For the convenience of the reader we provide the proof.
\begin{proof}[Proof of Corollary \ref{geo-cor}]
Let $X$ be as in Theorem \ref{geo-thm-intro} and let $C \in |L|$ be general. We have a surjective, finite morphism
$$d: \mathrm{Gr}_2(\mathrm{H}^0(E)) \to |L|$$
which takes a two dimensional subspace $W \seq \mathrm{H}^0(E)$ to the degeneracy locus of the evaluation morphism $W \otimes \mathcal{O}_X \xrightarrow{ev} E$, \cite{V1}. We may naturally identify the cokernel of $ev$ with $\mathrm{K}_C \otimes A^{-1}$, where $C=d(W) \in |L|$ and $A \in W^1_{k+1}(C)$, where $W^1_{k+1}(C)$ is the Brill--Noether locus of line bundles of degree $k+1$ with two sections, \cite{lazarsfeld-BNP}, \cite{aprodu-farkas}. If $C$ is sufficiently general, then $d^{-1}(C) \simeq W^1_{k+1}(C)$ is reduced, \cite{lazarsfeld-BNP}. Under the identification $d^{-1}(C) \simeq W^1_{k+1}(C)$, a line bundle $A \in W^1_{k+1}(C)$ is mapped to $\mathrm{H}^0(A)^{\vee}\seq \mathrm{H}^0(E)$. \smallskip

 By \cite[Proof of Prop.\ 7]{V1}, the spaces $\rm{Sym}^{k-2}\mathrm{H}^0(A)^{\vee}$, $A \in W^1_{k+1}(C)$ generate $\rm{Sym}^{k-2}\mathrm{H}^0(E)$. Each such $\mathrm{H}^0(A)^{\vee}$ corresponds to a line $T_A \seq \PP(\mathrm{H}^0(E))$ and set $\displaystyle{T:=\bigcup_{A \in W^1_{k+1}(C)} T_A}$ to be the union of these lines. The image of $T$ under $\widetilde{\psi}: \PP(\mathrm{H}^0(E)) \to \PP(\mathrm{K}_{k-1,1}(X,L))$ is non-degenerate so $\mathrm{K}_{k-1,1}(X,L)$ is generated by the subspaces $\mathrm{K}_{k-1,1}(X,L, \mathrm{H}^0(L \otimes I_{Z(s)})$, where $s \in \mathrm{H}^0(E)$ is a section corresponding to a $g^1_{k+1}$ on the fixed curve $C$. After restricting to $C$, such spaces are identified with subspaces of the form $$\mathrm{K}_{k-1,1}(C,\omega_C, \mathrm{H}^0(\omega_C \otimes A^{\vee})) \seq \mathrm{K}_{k-1,1}(C,\omega_C),\; \; A \in W^1_{k+1}(C)$$ under the isomorphism $\mathrm{K}_{k-1,1}(X,L) \simeq \mathrm{K}_{k-1,1}(C, \omega_C)$.
\end{proof}

\section{Voisin's Theorem in Odd Genus} \label{odd-genus}
In this section we prove Green's Conjecture for generic curves of odd genus $g=2k+1 \geq 9$. We broadly follow a similar strategy as \cite{V2} to reduce to the even genus case, however the Secant Bundle approach provides significant simplifications. As in Voisin's proof, we work with a K3 surface $X$ of Picard rank two over $\C$, with $\text{Pic}(X)$ generated by a big and nef line bundle $L'$ with $(L')^2=2k+2$ together with a smooth rational curve $\Delta$ with $(L' \cdot \Delta)=0$.\smallskip

The following lemma is adapted from \cite[Prop.\ 1]{V2}. 
\begin{lem} \label{BNP-gen}
 Set $L_n:=L'-n\Delta$ for $0 \leq n \leq 3$.  Then $L_n$ is base-point free and is further ample for $0<n \leq 2$. Furthermore, the linear system $|L_n|$ cannot be written as a sum of two pencils. In particular, smooth curves $C \in |L_n|$ are Brill--Noether general. If $C \in |L_n|$ is general then $C$ is, in addition, Petri general. 
\end{lem}
\begin{proof}
Since $(L_n)^2 \geq 0$ and $(L_n \cdot L')>0$, we see $\dim |L_n|>0$ by Riemann--Roch. We claim $L_n$ is nef, i.e.\ there is no rational, base component $R \sim aL'+b\Delta $ of $L_n$ with $(L_n \cdot R)<0$. Otherwise, $R$ and $L_n-R$ would be effective and $a \neq 0$ so $a =1$, but then $\dim |L_n|=\dim|L_n-R|=0$. Since $(\alpha \cdot \beta)$ is even for all $\alpha, \beta \in \mathrm{Pic}(X)$, there are no divisors $E$ with $(L_n \cdot E)=1$, hence the nef bundles $L_n$ are base-point free, \cite[Prop.\ 8]{mayer}. If $L_n$ is not ample for $n \neq 0$, there is a rational curve $R\sim aL'+b\Delta$ with $(L_n \cdot R)=0$. Then $b=\frac{-ag}{n}$ with $a>0$, so $(R)^2=-2a^2g(\frac{g}{n^2}-1)<-2$.

For the last claim it suffices to show that we cannot write $L_n=A_1+A_2$ for divisors $A_i$ with $h^0(A_i)\geq 2$, $i=1,2$, \cite{lazarsfeld-BNP}. Writing $A_i=a_i L'+b_i \Delta$, we must have $a_j=0$ for some $j\in \{1,2 \}$. But then $h^0(A_j)=1$ which is a contradiction.
\end{proof}

Setting $L:=L_1=L'-\Delta$, a general $C \in |L|$ is a curve of genus $g=2k+1$. To verify Green's conjecture for $C$, it suffices to show $\mathrm{H}^1(\wedge^{k+1}M_L)=\mathrm{K}_{k,1}(X,L)=0$. We have an exact sequence $0 \to M_L \to M_{L'} \to \mathcal{O}(-\Delta) \to 0$
of vector bundles on $X$. This gives an exact sequence $$0 \to \wedge^{k+1}M_L \to \wedge^{k+1}M_{L'} \to \wedge^k M_L (-\Delta) \to 0.$$
The induced map $pr_{k}: \; \mathrm{H}^1(\wedge^{k+1}M_{L'}) \to \mathrm{H}^1(\wedge^k M_L (-\Delta))$ is called the \emph{projection map}.
\begin{lem} \label{proj-crit}
Suppose the projection map $pr_k$ is injective. Then $\mathrm{H}^1(\wedge^{k+1}M_L)=0$.
\end{lem}\begin{proof}
Indeed $\mathrm{H}^0(\wedge^k M_L (-\Delta))\seq \mathrm{H}^0(\wedge^k M_L)=0$. 
\end{proof}

 We have a rational resolution of singularities $\mu : X \to \hat{X}$, contracting $\Delta$ to a du Val singularity $p$ on a nodal K3 surface $\hat{X}$. Then $\hat{X}$ admits a line bundle $\hat{L}$ with $\mu^*{\hat{L}}=L'$. We have $\mathrm{Pic}(\hat{X})=\mathrm{Cl}(\hat{X})\simeq \mathrm{Cl}(X\setminus \Delta)$, so  $\mathrm{Cl}(\hat{X})=\mathbb{Z}[\hat{L}]$ and $\hat{X}$ is factorial. Consider the rank two Lazarsfeld--Mukai bundle $\hat{E}$ on $\hat{X}$ induced by a $g^1_{k+2}$ on a general $C \in |\hat{L}|$. Set $E:=\mu^*\hat{E}$, which is a rank two bundle on $X$ induced by a $g^1_{k+2}$ on a general $C \in |L'|$. 
 
\begin{lem} \label{green-pic2}
We have $\mathrm{K}_{k+1,1}(X,L')=0$. Further, there is a natural isomorphism $$\mathrm{Sym}^{k-1}\mathrm{H}^0(X,E) \simeq \mathrm{K}_{k,1}(X,L').$$
\end{lem}
\begin{proof}
Set $\PP=\PP(\mathrm{H}^0(\hat{X}, \hat{E}))$ and let $\hat{\mathcal{Z}} \seq \hat{X} \times \PP$ be the locus defined by $\{ (x,s) \; | \; s(x)=0 \}$. Then $\hat{\mathcal{Z}}$ is a projective bundle over $\hat{X}$, is a local complete intersection in $\hat{X} \times \PP$ and is finite and flat over $\PP$ by ``miracle flatness'', \cite[Prop.\ 6.1.5]{EGA}. Note that, as in Section \ref{sec1}, finiteness of $\hat{\mathcal{Z}}$ over $\PP$ comes from the fact that $\mathrm{Cl}(\hat{X})=\mathbb{Z}[\hat{L}]$ and $h^0(\hat{E}(-\hat{L}))=h^0(\hat{E}^{\vee})=0$. The formal cohomological arguments from Section \ref{sec1} and Section \ref{sec2}, up to Corollary \ref{main-cor-nat-iso}, go through unchanged. Hence we see $\mathrm{K}_{k+1,1}(\hat{X}, \hat{L})=0$, and $\mathrm{Sym}^{k-1}\mathrm{H}^0(\hat{E}) \simeq \mathrm{K}_{k,1}(\hat{X},\hat{L})$. Since $\hat{X}$ has rational singularities, $\mathrm{H}^0(X,nL') \simeq \mathrm{H}^0(\hat{X}, n\hat{L})$ for all $n$, and the claim follows.
\end{proof}

We will prove that we have a natural injection $\mathrm{Sym}^{k-1}\mathrm{H}^0(X,E) \hookrightarrow \mathrm{H}^1(\wedge^k M_L (-\Delta))$. By the above lemma, this implies injectivity of $pr_k$, so that Voisin's Theorem follows. 
\begin{lem}[Voisin, \cite{V2}, Proposition 2] \label{decomposable}
The natural map $d: \wedge^2 \mathrm{H}^0(E) \to \mathrm{H}^0(\wedge^2 E)=\mathrm{H}^0(L')$ does not vanish on decomposable elements.
\end{lem}
\begin{proof}
Suppose $v_1, v_2 \in \mathrm{H}^0(E)$ with $v_1 \wedge v_2 \neq 0$, $d(v_1 \wedge v_2)=0$. Then $v_1, v_2$ generate a rank one subsheaf $H_1 \seq E$ with at least two sections. Let $H_2:=E/H_1$ and $M_2:=H_2/\mathrm{T}(H_2)$, where $\mathrm{T}(H_2)$ denotes the torsion subsheaf. We have a short exact sequence $0 \to M_1 \to E \to M_2 \to 0$, with $M_1,M_2$ rank-one, torsion-free sheaves and $h^0(M_1) \geq 2$. Then $M_i=N_i \otimes I_i$, for $N_i$ a line bundle and $I_i$ an ideal sheaf of a $0$-dimensional scheme for $i=1,2$. Since $E$ is globally generated with $h^2(E)=0$, we have $h^0(N_2) \geq 2$. But then $\det(E)=L'=N_1 \otimes N_2$ is a sum of line bundles with at least two sections. We have already seen that this cannot occur.
\end{proof}
To set things up, we need a lemma.
\begin{lem}
The bundle $E(-\Delta)$ is isomorphic to the Lazarsfeld--Mukai bundle $F$ corresponding to a $g^1_{k}$ on a general $C' \in |L-\Delta|$. In particular, $E(-\Delta)$ is globally generated.
\end{lem}
\begin{proof}
We claim that $F$ is $\mu$-stable with respect to $L-\Delta$. Otherwise, we have a filtration
$$0 \to M \to F \to N \otimes I_{\zeta} \to 0$$
where $M, N$ are line bundles, $I_{\zeta}$ is the ideal sheaf of a $0$-dimensional scheme of length $k-(M \cdot N)$, where $h^0(N) \geq 2$ and with $\mu(M) \geq \mu(F)=g-4 \geq \mu(N)$, cf.\ \cite{lelli-chiesa}. We have $h^2(M)=0$ since $\mu(M)>0$ and further $(M)^2=\mu(M)-(M\cdot N)\geq (g-4)-k \geq 0$. Thus $h^0(M) \geq 2$ by Riemann--Roch, contradicting that $L-\Delta=\det(F)$ cannot be written as a sum of two pencils.
\smallskip

The rank two bundle $E(-\Delta)$ is simple with $\det E(-\Delta)=L-\Delta$ and $\chi(E(-\Delta))=k+1$. We claim that $E(-\Delta)$ is stable with respect to $L-\Delta$. Since $F$ is the unique stable bundle with these invariants, $E(-\Delta) \simeq F$. If $E(-\Delta)$ is not stable, choose a filtration $0 \to M' \to E(-\Delta) \to N' \otimes I_{\zeta'} \to 0$
as above. We again have $h^0(M') \geq 2$, and since $N'(\Delta) \otimes I_{\zeta'}$ is a quotient of the globally generated bundle $E$ and $h^2(E)\neq 0$, $h^0(N'(\Delta)) \geq 2$. This contradicts that $L$ cannot be written as a sum of two pencils.
\end{proof}

Following \cite[pg.\ 1177]{V2}, let $N$ denote the rank $k-1$ kernel bundle fitting into the sequence
$$0 \to N \to \mathrm{H}^0(E(-\Delta))\otimes \mathcal{O}_X \to E(-\Delta) \to 0.$$
For any $t \in \mathrm{H}^0(E)$, we have a map $\wedge \, t :\mathrm{H}^0(E(-\Delta)) \to \mathrm{H}^0(\wedge^2 E (-\Delta))=\mathrm{H}^0(L)$, inducing a map $N \to M_L$. This globalizes to a vector bundle morphism
$$r: \; N  \boxtimes \mathcal{O}_{\PP}(-1) \hookrightarrow \mathcal{M} ,$$
on $X \times \PP$, where $\PP:=\PP(\mathrm{H}^0(X,E)) \simeq \PP^{k+2}$, where $p: X \times \PP \to X$, $q: X \times \PP \to \PP$ are the projections and $\mathcal{M}:=p^*M_L$.
\begin{remark}
We may compare the definition of $N$ to the vector bundles defined in Section \ref{sec1}. Let $(X,L)$ be a polarized K3 surface of even genus $g=2k$. Then, for any $t \in \mathrm{H}^0(X,E)$ we have an isomorphism between the secant sheaf $\mathcal{S}'_t$ and
$$N_E:= \mathrm{Ker}(\mathrm{H}^0(E) \otimes \mathcal{O}_X \twoheadrightarrow E),$$
as follows readily from the exact sequence $0 \to \mathcal{O}_X \xrightarrow{t} E \xrightarrow{\wedge t} I_{Z(t)} \otimes L \to 0$. Notice that the isomorphism $\mathcal{S}'_t \simeq N_E$ depends on $t$. 
\end{remark}

From Lemma \ref{decomposable}, $r$ fails to be injective \emph{on fibres} precisely at points in the locus 
$$\mathcal{Z}:=\left \{ (x,s) \in X \times \PP\left(\mathrm{H}^0(E(-\Delta))\right) \; | \; s(x)=0 \right\}.$$ 
Thus $\mathrm{Coker}(r)$ fails to be locally free along $\mathcal{Z} \seq X \times \Lambda$, where $\Lambda:=\PP\left(\mathrm{H}^0(E(-\Delta))\right) \seq \PP$. \\

We rectify the failure of $\mathrm{Coker}(r)$ to be locally free through a standard construction. Define $\pi: B \to X \times \PP$ as the blow-up along the codimension four locus $\mathcal{Z}$ and let $p':=p\circ \pi$, $q':=q \circ \pi$. Let $j: D \hookrightarrow B$ be the exceptional divisor. For any vector bundle $A$ on $X \times \PP$,  \begin{align} \label{mult-D-iso}
\mathrm{H}^j(B,\pi^*A (nD))\simeq \mathrm{H}^j(B,\pi^*A )\end{align} for any $j$ and for $1 \leq n \leq 3$.\smallskip

For $p=(x,t) \in \mathcal{Z}$, the kernel of $r \otimes k(p)$ is isomorphic to $\C\langle t\rangle \seq \mathrm{H}^0(X,E)$. Thus $$\mathrm{Ker}(r_{|_{\mathcal{Z}}})\simeq ({q}^*\mathcal{O}_{\PP}(-2))_{|_{\mathcal{Z}}},$$ where the inclusion $({q}^*\mathcal{O}_{\PP}(-2))_{|_{\mathcal{Z}}} \hookrightarrow (N  \boxtimes \mathcal{O}_{\PP}(-1))_{|_{\mathcal{Z}}}$ is given by the section $u \in \mathrm{H}^0(\mathcal{Z},N  \boxtimes \mathcal{O}_{\PP}(1))\seq \mathrm{H}^0(E) \otimes \mathrm{H}^0({q}^*\mathcal{O}_{\PP}(1)_{|_{\mathcal{Z}}})$ obtained by restricting $id \in \mathrm{H}^0(E) \otimes \mathrm{H}^0({q}^*\mathcal{O}_{\PP}(1))$ to $\mathcal{Z}$.\smallskip

We now perform an \emph{elementary transformation} on  $N  \boxtimes \mathcal{O}_{\PP}(-1)$. Define $S$ as the dual bundle to $\mathrm{Im}({\pi^*r}^{\vee})$. Then $S^{\vee}$ is a vector bundle of rank $k-1$ fitting into the exact sequence
$$0 \to S^{\vee} \to \pi^* (N^{\vee} \boxtimes \mathcal{O}_{\PP}(1)) \to {{q'}^* \mathcal{O}_{\PP}(2)}_{|_D} \to 0.$$ 
From the definition of $S$, we have an exact sequence $$0 \to S \to \pi^* \mathcal{M} \to \Gamma \to 0,$$ where $\Gamma$ is locally free of rank $k+2$. 

The blow-up of a projective space $\PP(V)$ along a subspace $W \seq V$ is a projective bundle over $\PP(V/W)$, \cite[\S 9.3.2]{3264}. Thus $B$ is a projective bundle $\PP(\mathcal{H})$ over $\PP(\mathcal{F})$, where $\mathcal{F}:=\mathrm{Coker}(N \hookrightarrow \mathrm{H}^0(E) \otimes \mathcal{O}_X)$. 
The projection morphism $\displaystyle{\chi: B \to \PP(\mathcal{F})}$ is defined over $X$ with $$\chi^* \mathcal{O}_{\PP(\mathcal{F})}(1) \simeq {q'}^*\mathcal{O}_{\PP}(1)(-D).$$ To describe $\mathcal{H}$, let $f: \PP(\mathcal{F}) \to X$ be the projection and define $\mathcal{P}$ by the exact sequence $0 \to \mathcal{P}^{\vee} \to f^*f_* \mathcal{O}_{\PP(\mathcal{F})}(1) \to \mathcal{O}_{\PP(\mathcal{F})}(1) \to 0.$ We have a surjection $\mathrm{H}^0(E) \otimes \mathcal{O}_{\PP(\mathcal{F})} \to \mathcal{P}$. The rank $k$ bundle $\mathcal{H}$ is defined by the exact sequence
$$0 \to \mathcal{H} \to \mathrm{H}^0(E) \otimes \mathcal{O}_{\PP(\mathcal{F})}\to \mathcal{P} \to 0.$$
We have a short exact sequence
$$ 0 \to f^* N \to \mathcal{H} \to \mathcal{O}_{\PP(\mathcal{F})}(-1) \to 0.$$
The isomorphism $B \simeq \PP(\mathcal{H})$ gives an identification $\mathcal{O}_{\PP(\mathcal{H})}(1) \simeq {q'}^*\mathcal{O}_{\PP}(1)$, \cite[Prop.\ 9.11]{3264}.
In summary, we have the following commutative diagram on $\PP(\mathcal{F})$:
{\small{$$\begin{tikzcd}
& 0 \arrow[d] & 0 \arrow[d]  &  \\
0 \arrow[r] & f^* N \arrow[r, "="] \arrow[d] & f^* N   \arrow[d] &   &\\
0 \arrow[r] & \mathcal{H}  \arrow[r] \arrow[d] & \mathrm{H}^0(E) \otimes \mathcal{O}_{\PP(\mathcal{F})}  \arrow[r] \arrow[d]  &  \mathcal{P} \ \arrow[r] \arrow[d, "="]& 0\\
0 \arrow[r] & \mathcal{O}_{\PP(\mathcal{F})}(-1)  \arrow[r] \arrow[d] & f^*\mathcal{F}=f^*(f_* \mathcal{O}_{\PP(\mathcal{F})}(1))^{\vee}  \arrow[r] \arrow[d]  &  \mathcal{P}  \arrow[r]  & 0\\
& 0 & 0 & &
\end{tikzcd},$$}}
Note that we are using the geometric notation for projective bundles, i.e.\ $\PP(\mathcal{F})=\mathrm{Proj}(\mathcal{F}^{\vee})$, and $\mathcal{F}^{\vee}=f_*\mathcal{O}_{\PP(\mathcal{F})}(1)$.
\begin{lem}
We have  $S \simeq T_{\chi} \otimes {q'}^*\mathcal{O}_{\PP}(-2)$, where $T_{\chi}$ is the relative tangent bundle.
\end{lem}
\begin{proof}
We keep track of the maps defined above in the commutative diagram
{\small{$$\begin{tikzcd}
B \arrow[rr, "\chi"] \arrow[dr,"p \circ \pi"'] & & \PP(\mathcal{F}) \arrow[dl, "f"] \\
& X &
\end{tikzcd}.$$}}
We have the relative Euler sequence $0 \to \mathcal{O}_{\PP(\mathcal{H})} \to \mathcal{O}_{\PP(\mathcal{H})}(1) \otimes \chi^* \mathcal{H} \xrightarrow{\alpha} T_{\chi} \to 0.$
Twisting by $\mathcal{O}_{\PP(\mathcal{H})}(-2)$, we have the composite map
$\pi^*(N \boxtimes \mathcal{O}_{\PP}(-1)) \to \chi^*\mathcal{H} \otimes {q'}^*\mathcal{O}_{\PP}(-1)\xrightarrow{\alpha} T_{\chi}(-2).$
Dualizing, we obtain an exact sequence
$$0 \to \Omega_{\chi}\otimes {q'}^*\mathcal{O}_{\PP}(2) \to \pi^*(N^{\vee} \boxtimes \mathcal{O}_{\PP}(1)) \to {{q'}^*\mathcal{O}_{\PP}(2)}_{|_D} \to 0,$$
and comparison with the defining sequence for $S^{\vee}$ gives $S^{\vee} \simeq  \Omega_{\chi}\otimes {q'}^*\mathcal{O}_{\PP}(2)$.
\end{proof}

We will need the following computation:
\begin{lem} \label{det-gamma}
We have $\det \Gamma \simeq {q'}^* \mathcal{O}_{\PP}(k-1)(-D-{p'}^*\Delta)$. In particular, we have a natural isomorphism $\mathrm{H}^2(B, \wedge^{k+2} \Gamma ({p'}^*\Delta)(D))\simeq \mathrm{Sym}^{k-1}\mathrm{H}^0(X,E)^{\vee}$.
\end{lem}
\begin{proof}
We have $\det \Gamma \simeq {p'}^*(L^{\vee}) \otimes \det S^{\vee}$ so $\det \Gamma \simeq {q'}^* \mathcal{O}_{\PP}(k-1)(-D-{p'}^*\Delta)$ and $\wedge^{k+2} \Gamma ({p'}^*\Delta)\simeq {q'}^* \mathcal{O}_{\PP}(k-1)(-D)$. Thus $\mathrm{H}^2(B, \wedge^{k+2} \Gamma ({p'}^*\Delta)(D))\simeq \mathrm{H}^2({q'}^* \mathcal{O}_{\PP}(k-1))\simeq \mathrm{H}^2(B,{q'}^* \mathcal{O}_{\PP}(k-1)) \simeq \mathrm{Sym}^{k-1}\mathrm{H}^0(X,E)^{\vee}$.
\end{proof}
\smallskip

We have natural isomorphisms $$\mathrm{K}_{k-1,1}(X,-\Delta, L)\simeq \mathrm{H}^1(\wedge^k M_L(-\Delta))\simeq \mathrm{H}^0(\wedge^{k-1}M_L (L-\Delta)),$$
see \cite{lazarsfeld-VBT} or \cite[\S 3]{ein-lazarsfeld-asymptotic}. We have $\mathrm{H}^0(\wedge^{k-1}M_L (L-\Delta))^{\vee} \simeq \mathrm{H}^2(\wedge^{k+2}M_L(\Delta))$, since $\wedge^{k-1}M_L^{\vee} \simeq \wedge^{k+2}M_L(L)$.\smallskip


Taking exterior powers of $0 \to S \to \pi^* \mathcal{M} \to \Gamma \to 0$ and twisting by $\mathcal{O}_{B}(D+{p'}^*\Delta)$
induces
$$\phi \; : \; \mathrm{H}^2(\wedge^{k+2}M_L(\Delta)) \to \mathrm{H}^2(\wedge^{k+2} \Gamma ({p'}^*\Delta)(D))\simeq \mathrm{Sym}^{k-1}\mathrm{H}^0(E)^{\vee}.$$
We will show that $\phi$ is surjective. We write $\mathcal{O}_B(j)$ for ${q'}^*\mathcal{O}_{\PP}(j) $. Consider the exact sequence
{\small{$$\ldots \to \bigwedge^{k+1}\pi^* \mathcal{M} \otimes S({p'}^*\Delta)(D)\to  \bigwedge^{k+2}\pi^* \mathcal{M} ({p'}^*\Delta)(D) \to \bigwedge^{k+2} \Gamma ({p'}^*\Delta)(D) \to 0.$$}}
The following proposition should be compared to \cite[Lemma 6]{V2}.
\begin{prop} \label{most-of-van}
We have $\mathrm{H}^{2+i}(\bigwedge^{k+2-i} \pi^* \mathcal{M} \otimes \mathrm{Sym}^i(S) ({p'}^*\Delta)(D))=0$ for $1 \leq i \leq k+2$, $i \neq k, k+1$.
\end{prop}
\begin{proof}
The Euler sequence $0 \to \mathcal{O}_B(-2) \to \chi^* \mathcal{H} (-1)  \to S \to 0$
gives exact sequences
$$0 \to \chi^* \mathrm{Sym}^{i-1} \mathcal{H} (-i-1) \to \chi^*\mathrm{Sym}^{i} \mathcal{H} (-i) \to \mathrm{Sym}^i(S) \to 0.$$
 We first claim
  $$\mathrm{H}^{2+i}(\bigwedge^{k+2-i} \pi^* \mathcal{M} \otimes \chi^*\mathrm{Sym}^{i} \mathcal{H} (-i)({p'}^*\Delta)(D))=0, \; \; \text{for $1 \leq i \leq k-1$}.$$
 The fibres of $\chi$ are projective spaces of dimension $k-1$ and $D$ has degree one on these fibres. Hence $R^j \chi_*\mathcal{O}_B(-i)(D)=0$ for all $j$ and $2 \leq i \leq k-1$. So the claim holds for $2 \leq i \leq k-1$ by the Leray spectral sequence. \smallskip 
 
 For $i=1$, the claim states $\mathrm{H}^3(\bigwedge^{k+1} \pi^* \mathcal{M} \otimes \chi^*\mathcal{H} (-1)({p'}^*\Delta)(D))=0$. We have $\omega_B \simeq {q'}^* \omega_{\PP}(3D)$, so that this is equivalent to  
$$\mathrm{H}^{k+1}(\bigwedge^{k+1} \pi^* \mathcal{M}^{\vee} \otimes \chi^*\mathcal{H}^{\vee}\otimes {q'}^*(\omega_{\PP}(1)) (-{p'}^*\Delta)(2D))=0$$
The claim now follows from the exact sequence $ 0 \to \mathcal{O}_{B}(1)(-D) \to \chi^* \mathcal{H}^{\vee} \to {p'}^*N^{\vee} \to 0$ since
{\small{\begin{align*}
&\mathrm{H}^{k+1}(B, \bigwedge^{k+1} \pi^* \mathcal{M}^{\vee} \otimes {q'}^*(\omega_{\PP}(2)) (-{p'}^*\Delta)(D)) \simeq \mathrm{H}^{k+1}(X\times \PP, \bigwedge^{k+1} M_L^{\vee}(-\Delta)\boxtimes \omega_{\PP}(2))=0, \; \; \text{and} \\
&\mathrm{H}^{k+1}(B,\bigwedge^{k+1} \pi^* \mathcal{M}^{\vee} \otimes {p'}^*N^{\vee} \otimes {q'}^*(\omega_{\PP}(1)) (-{p'}^*\Delta)(2D))  \simeq \mathrm{H}^{k+1}(X \times \PP, \bigwedge^{k+1} M_L^{\vee}\otimes N^{\vee}(-\Delta) \boxtimes \omega_{\PP}(1)),
\end{align*}}}
by the K\"unneth formula.
\smallskip

To finish the proof in the case $1 \leq i \leq k-1$,  it suffices to note that
 $$\mathrm{H}^{3+i}(\bigwedge^{k+2-i} \pi^* \mathcal{M} \otimes \chi^*\mathrm{Sym}^{i-1} \mathcal{H} (-i-1)({p'}^*\Delta)(D))=0, \; \; \text{for $1 \leq i \leq k-1$},$$ is immediate as above. We are now left with the case $i=k+2$. The inclusion $$\pi^*( \mathrm{Sym}^{k+2}N(\Delta) \boxtimes \mathcal{O}_{\PP}(-(k+2))) \hookrightarrow \mathrm{Sym}^{k+2}(S)({p'}^*\Delta)$$ is an isomorphism off $D$. Since $\dim D=k+3$, it suffices to show $$\mathrm{H}^{k+4}(X \times \PP^{k+2}, \mathrm{Sym}^{k+2}N(\Delta) \boxtimes \mathcal{O}_{\PP}(-(k+2)))=0.$$ This follows from the K\"unneth formula. \end{proof}

We next show that we continue to have the above vanishing in the case $i=k+1$. The following lemma, stated in \cite{V1}, Proof of Prop.\ 6, was explained to us by C.\ Voisin.
\begin{lem} \label{LM-mult-map}
The multiplication map $\mathrm{H}^0(X,E(-\Delta)) \otimes \mathrm{H}^0(X,L-\Delta) \to \mathrm{H}^0(X,E(L-2\Delta))$ is surjective.
\end{lem}
\begin{proof}
Let $C \in |L-\Delta|$ be a smooth curve. It suffices to prove surjectivity of the restricted multiplication map $\mathrm{H}^0(E_{|_C}(-\Delta)) \otimes \mathrm{H}^0(\mathrm{K}_C) \to \mathrm{H}^0(E_{|_C}(\mathrm{K}_C-\Delta))$, \cite[Observation 2.3]{gallego-purnaprajna}. Now $\mathrm{H}^0(E_{|_C}(-\Delta)) \simeq \mathrm{H}^0(B) \oplus \mathrm{H}^0(\mathrm{K}_C-B)$ where $B$ is a $g^1_k$ on $C$, \cite{voisin-wahl}. The statement now follows from the following: for any base-point free line bundle $A$ on $C$ with $h^0(A) \geq 2$, the multiplication map $\mathrm{H}^0(\mathrm{K}_C) \otimes \mathrm{H}^0(A) \to \mathrm{H}^0(\mathrm{K}_C+A)$ is surjective. If $h^0(A)=2$, this follows immediately from the base-point free pencil trick. Indeed, we have an exact sequence
$$0 \to K_C \otimes A^{-1} \to H^0(A) \otimes K_C \to K_C \otimes A \to 0,$$
and the claim follows from fact that $h^1(K_C \otimes A^{-1})=h^1(H^0(A) \otimes K_C )=h^0(A)$. Otherwise, let $Z$ be a general effective divisor of degree $h^0(A)-2$. Thus $\mathrm{H}^0(\mathrm{K}_C) \otimes \mathrm{H}^0(A-Z) \to \mathrm{H}^0(\mathrm{K}_C+A-Z)$ is surjective. Since this holds for any general such divisor, this proves the claim.
\end{proof}
We will make use of the following direct consequence of the previous lemma.
\begin{lem} \label{mult-maps}
We have $\mathrm{H}^2(X,M_L(\Delta) \otimes N)=\mathrm{H}^2(M_L(\Delta))=0$.
\end{lem}
\begin{proof}
From the exact sequence
$$0 \to N \to \mathrm{H}^0(E(-\Delta)) \otimes \mathcal{O}_X \to E(-\Delta)\to 0,$$
it suffices to show $\mathrm{H}^2(M_L(\Delta))=0$ and $\mathrm{H}^1(M_L(E))=0$. The first vanishing follows immediately from the defining sequence for $M_L$.\smallskip

The vanishing $\mathrm{H}^1(M_L(E))=0$ is equivalent to surjectivity of the multiplication map $\mathrm{H}^0(L) \otimes \mathrm{H}^0(E) \to \mathrm{H}^0(E(L))$. We will use the following principle. If we have a commutative diagram
{\small{$$\begin{tikzcd}
0 \arrow[r] & V_1 \arrow[r, "\alpha"]  \arrow[d, "f"] & V_2 \arrow[r, "\beta"] \arrow[d, "g"] & V_3 \arrow[r] \arrow[d,"h"] & 0  \\
0 \arrow[r] & W_1 \arrow[r, "\gamma"] & W_2 \arrow[r, "\delta"] & W_3 
\end{tikzcd}$$}}
of vector spaces with exact rows, and if $h$ is surjective, then so is $\delta$. If, further, $f$ is surjective, then so is $g$.\smallskip

We have surjectivity of $\mathrm{H}^0(L-\Delta) \otimes \mathrm{H}^0(E(-\Delta)) \to \mathrm{H}^0(E(L-2\Delta))$ by Lemma \ref{LM-mult-map}.  For a \emph{general} $s \in \mathrm{H}^0(E)$, we have the exact sequence $0 \to \mathcal{O}_X \to E \to L'\otimes I_{Z(s)} \to 0$. This implies $E_{|_{\Delta}} \simeq \mathcal{O}_{\Delta}^{\oplus 2}$ is trivial. Thus $\mathrm{H}^0(L-\Delta) \otimes \mathrm{H}^0(E_{|_{\Delta}}) \twoheadrightarrow \mathrm{H}^0(E_{|_{\Delta}}(L-\Delta))$. This shows $\mathrm{H}^0(L-\Delta) \otimes \mathrm{H}^0(E) \twoheadrightarrow \mathrm{H}^0(E(L-\Delta))$. Since $\mathrm{H}^0(L_{|_{\Delta}}) \otimes \mathrm{H}^0(E) \twoheadrightarrow \mathrm{H}^0(E_{|_{\Delta}}(L))$, using the triviality of $E_{|_{\Delta}}$, we obtain the vanishing $\mathrm{H}^1(M_L(E))=0$.\end{proof}

We now prove the vanishing in case $i=k+1$. The following proof should be compared to \cite[Lemma 8]{V2}.
\begin{prop} \label{green-odd-main}
We have $\mathrm{H}^{k+3}(B,\pi^* \mathcal{M} \otimes \mathrm{Sym}^{k+1}(S) ({p'}^*\Delta)(D))=0$. 
 \end{prop}
\begin{proof}
We have the short exact sequence
$$0 \to \chi^* \mathrm{Sym}^{k} \mathcal{H} (-k-2) \to \chi^*\mathrm{Sym}^{k+1} \mathcal{H} (-k-1) \to \mathrm{Sym}^{k+1}(S) \to 0.$$
We firstly claim that $$\mathrm{H}^{k+3}(\pi^* \mathcal{M} \otimes  \chi^* \mathrm{Sym}^{k} \mathcal{H} (-k-2)  ({p'}^*\Delta)(D))\to \mathrm{H}^{k+3}(\pi^* \mathcal{M} \otimes  \chi^*\mathrm{Sym}^{k+1} \mathcal{H} (-k-1)  ({p'}^*\Delta)(D))$$
is surjective. We have $\mathcal{O}(D) \simeq \mathcal{O}(1)\otimes \chi^*\mathcal{O}_{\PP(\mathcal{F})}(-1)$ and further $\omega_{\chi}\simeq \chi^*\det \mathcal{H}^{\vee}(-k)$. The map can thus be written as
{\small{\begin{align*}
&\mathrm{H}^4(\PP(\mathcal{F}),f^*M_L(\Delta) \otimes \mathrm{Sym}^k\mathcal{H} \otimes R^{k-1}\chi_*\mathcal{O}(-(k+1))\otimes \mathcal{O}_{\PP(\mathcal{F})}(-1)) \to  \\
&\mathrm{H}^4(\PP(\mathcal{F}),f^*M_L(\Delta) \otimes \mathrm{Sym}^{k+1}\mathcal{H} \otimes R^{k-1}\chi_*\mathcal{O}(-k)\otimes \mathcal{O}_{\PP(\mathcal{F})}(-1)) .
\end{align*}}}
Using relative duality, this becomes
{\small{\begin{align*}
&\mathrm{H}^4(\PP(\mathcal{F}),f^*M_L(\Delta) \otimes \mathrm{Sym}^k\mathcal{H} \otimes \mathcal{H} \otimes \mathcal{O}_{\PP(\mathcal{F})}(-1)\otimes \det \mathcal{H}) \to  \\
&\mathrm{H}^4(\PP(\mathcal{F}),f^*M_L(\Delta) \otimes \mathrm{Sym}^{k+1}\mathcal{H} \otimes \mathcal{O}_{\PP(\mathcal{F})}(-1)\otimes \det \mathcal{H})),
\end{align*}}}
since $\chi_*\mathcal{O}(n)\simeq \mathrm{Sym}^n\mathcal{H}^{\vee}$. This map is surjective, since the composite $ \mathrm{Sym}^{k+1}\mathcal{H} \to \mathrm{Sym}^k\mathcal{H} \otimes \mathcal{H} \to \mathrm{Sym}^{k+1}\mathcal{H} $ of natural maps is multiplication by $k+1$.\smallskip

To conclude, it suffices that $\mathrm{H}^{k+4}(\pi^* \mathcal{M} \otimes  \chi^* \mathrm{Sym}^{k} \mathcal{H} (-k-2)({p'}^*\Delta)(D))=0$, or, equivalently
$$\mathrm{H}^5(\PP(\mathcal{F}),f^*M_L(\Delta) \otimes \mathrm{Sym}^k\mathcal{H} \otimes \mathcal{H} \otimes \mathcal{O}_{\PP(\mathcal{F})}(-1)\otimes \det \mathcal{H}))=0.$$
The same argument identifies this space with $\mathrm{H}^{k+4}(\pi^* \mathcal{M} \otimes \chi^*\mathcal{H} (-2k-1+{p'}^*\Delta)(D))$, using $\chi_*\mathcal{O}(n)\simeq \mathrm{Sym}^n\mathcal{H}^{\vee}$ again. From the exact sequence
$$0 \to {p'}^*N \to \chi^* \mathcal{H} \to {q'}^*\mathcal{O}_{\PP}(-1)(D) \to 0,$$
the vanishing is implied by
{\small{\begin{align*}
\mathrm{H}^{k+4}(X \times \PP, M_L(\Delta) \otimes N \boxtimes \mathcal{O}_{\PP}(-2k-1))=\mathrm{H}^{k+4}(X \times \PP, M_L(\Delta)\boxtimes \mathcal{O}_{\PP}(-2k-2))=0
\end{align*}}}
which follow from Lemma \ref{mult-maps} and the K\"unneth formula.
\end{proof}

From the above vanishings, to show the surjectivity of $\phi$ it is now enough to show that the natural map $\mathrm{H}^{k+2}(\pi^*\mathcal{M} \otimes \mathrm{Sym}^{k+1}(S)({p'}^*\Delta)(D)) \to \mathrm{H}^{k+2}(\bigwedge^2 \pi^*\mathcal{M} \otimes \mathrm{Sym}^{k}(S)({p'}^*\Delta)(D))$ is surjective. We first rewrite this map.
\begin{lem} \label{identilem}
For any $j$, we have have natural isomorphisms 
{\small{
\begin{align*}
\mathrm{H}^j(B, \pi^* \mathcal{M} \otimes \mathrm{Sym}^{k+1}(S)({p'}^*\Delta)(D))&\simeq \mathrm{H}^{j+1}(B, \pi^* \mathcal{M} \otimes S(-2k)({p'}^*\Delta)(D)), \\
\mathrm{H}^{j}(B,\bigwedge^{2}\pi^*\mathcal{M}  \otimes \mathrm{Sym}^k(S)({p'}^*\Delta)(D))&\simeq \mathrm{H}^{j+1}(B,\bigwedge^{2}\pi^*\mathcal{M}(-2k)({p'}^*\Delta)(D))
\end{align*}
}}
\end{lem}
\begin{proof}
Arguing as in Proposition \ref{green-odd-main}, identify $\mathrm{H}^j(B, \pi^* \mathcal{M} \otimes \mathrm{Sym}^{k+1}(S)({p'}^*\Delta)(D))$ with
{\small{\begin{align*}
\mathrm{Ker}(&\mathrm{H}^{j+2-k}(\PP(\mathcal{F}),f^*M_L(\Delta) \otimes \mathrm{Sym}^k\mathcal{H} \otimes \mathcal{H} \otimes \mathcal{O}_{\PP(\mathcal{F})}(-1) \otimes \det\mathcal{H}) \twoheadrightarrow  \\
&\mathrm{H}^{j+2-k}(\PP(\mathcal{F}),f^*M_L(\Delta) \otimes \mathrm{Sym}^{k+1}\mathcal{H} \otimes \mathcal{O}_{\PP(\mathcal{F})}(-1)) \otimes \det \mathcal{H}) .
\end{align*}}}
The above surjection splits naturally, identifying $\mathrm{H}^j(B, \pi^* \mathcal{M} \otimes \mathrm{Sym}^{k+1}(S)({p'}^*\Delta)(D))$ with
{\small{\begin{align*}
\mathrm{Coker}(&\mathrm{H}^{j+2-k}(\PP(\mathcal{F}),f^*M_L(\Delta) \otimes \mathrm{Sym}^{k+1}\mathcal{H} \otimes \mathcal{O}_{\PP(\mathcal{F})}(-1) \otimes \det \mathcal{H}) \hookrightarrow \\
&\mathrm{H}^{j+2-k}(\PP(\mathcal{F}),f^*M_L(\Delta) \otimes \mathrm{Sym}^k\mathcal{H} \otimes \mathcal{H} \otimes \mathcal{O}_{\PP(\mathcal{F})}(-1)) \otimes \det \mathcal{H}),
\end{align*}}}
which is naturally identified with
\begin{align*}
\mathrm{Coker}\left( \mathrm{H}^{j+1}(B, \pi^* \mathcal{M}(-2k-2)({p'}^*\Delta)(D)) \hookrightarrow
 \mathrm{H}^{j+1}(B, \pi^* \mathcal{M} \otimes \mathcal{H}(-2k-1)({p'}^*\Delta)(D))\right).
\end{align*}
From the sequence $0 \to \mathcal{O}_B(-2) \to \chi^* \mathcal{H} (-1)  \to S \to 0$, the above space is isomorphic to $\mathrm{H}^{j+1}(B, \pi^* \mathcal{M} \otimes S(-2k)({p'}^*\Delta)(D))$, as required.\smallskip

For the second isomorphism, the exact sequence
$$0 \to \chi^* \mathrm{Sym}^{k-1} \mathcal{H} (-k-1) \to \chi^*\mathrm{Sym}^{k} \mathcal{H} (-k) \to \mathrm{Sym}^k(S) \to 0,$$
together with relative duality for $\chi$, provides natural isomorphisms 
{\small{\begin{align*}
\mathrm{H}^{j}(B, \bigwedge^{2}\pi^*\mathcal{M}  \otimes \mathrm{Sym}^k(S)({p'}^*\Delta)(D))& \simeq \mathrm{H}^{j+1}(B,\bigwedge^{2}\pi^*\mathcal{M}  \otimes \chi^*\mathrm{Sym}^{k-1}(\mathcal{H})(-k-1)({p'}^*\Delta)(D)) \\
&\simeq \mathrm{H}^{k+1-j}(\PP(\mathcal{F}), \bigwedge^{2}f^*M_L(\Delta) \otimes \mathrm{Sym}^{k-1}\mathcal{H} \otimes \mathcal{O}_{\PP(\mathcal{F})}(-1)\otimes \det \mathcal{H})\\
& \simeq \mathrm{H}^{j+1}(B,\bigwedge^{2}\pi^*\mathcal{M}  \otimes {q'}^*\mathcal{O}_{\PP}(-2k)({p'}^*\Delta)(D)).
\end{align*}}}
\end{proof}
In order to prove surjectivity of the natural map $$\mathrm{H}^{k+3}(\pi^* \mathcal{M} \otimes S(-2k)({p'}^*\Delta)(D)) \to \mathrm{H}^{k+3}(\bigwedge^{2}\pi^*\mathcal{M}(-2k)({p'}^*\Delta)(D)),$$ we follow a strategy reminiscent of \cite[\S 3, Fourth step]{V2}. Consider the exact sequence
$$0 \to \mathrm{Sym}^2(S) \to \pi^*\mathcal{M} \otimes S \to \bigwedge^2 \pi^*\mathcal{M} \to \bigwedge^2 \Gamma \to 0.$$
We have $\mathrm{H}^{k+4}(\pi^* \mathcal{M} \otimes S(-2k)({p'}^*\Delta)(D))=0$ by Lemma \ref{identilem} and Proposition \ref{green-odd-main}. Thus the surjectivity of $\phi$ will follow from
$$\mathrm{H}^{k+4}(\mathrm{Sym}^2(S)(-2k)({p'}^*\Delta)(D))=\mathrm{H}^{k+3}(\bigwedge^2\Gamma (-2k)({p'}^*\Delta)(D))=0.$$
\begin{lem}
We have $\mathrm{H}^{k+4}(\mathrm{Sym}^2(S)(-2k)({p'}^*\Delta)(D))=0$.
\end{lem}
\begin{proof}
The inclusion $\pi^*( \mathrm{Sym}^{2}N(\Delta) \boxtimes \mathcal{O}_{\PP}(-2)) \hookrightarrow \mathrm{Sym}^{2}(S)({p'}^*\Delta)$ is an isomorphism off $D$, so that it suffices to show $\mathrm{H}^{k+4}(X \times \PP, \mathrm{Sym}^{2}N(\Delta) \boxtimes \mathcal{O}_{\PP}(-2-2k))=0$. By the K\"unneth formula, it in turn suffices to show $\mathrm{H}^2(X,\mathrm{Sym}^{2}N(\Delta))=0$. We have an exact sequence
$$0 \to  \mathrm{Sym}^2N \to \mathrm{Sym}^2 \mathrm{H}^0(E(-\Delta)) \otimes \mathcal{O}_X \to \mathrm{H}^0(E(-\Delta)) \otimes E(-\Delta) \to L-\Delta \to 0.$$
Since $\mathrm{H}^2(X,\mathcal{O}_X(\Delta))=H^1(X,E)=0$, it is enough to show that the determinant map
$$\mathrm{det} : \mathrm{H}^0(E(-\Delta)) \otimes \mathrm{H}^0(E) \to \mathrm{H}^0(L)$$
is surjective. Let $C \in |L+\Delta|$ be general (in particular $C$ is disjoint from $\Delta$) and let $A$ be any $g^1_{k+2}$ on $C$. We have an exact sequence
$$0 \to \mathrm{H}^0(A)^{\vee} \otimes \mathcal{O}_X \to E \to \omega_C \otimes A^{\vee} \to 0,$$
\cite{lazarsfeld-BNP}. Restricting to $C$ we have isomorphisms $\mathrm{H}^0(E)=\mathrm{H}^0(E_{|_C})\simeq \mathrm{H}^0(C,A)\oplus \mathrm{H}^0(\omega_C \otimes A^{\vee})$, \cite{voisin-wahl}. Twisting the above exact sequence by $-\Delta$, we have a natural surjection $\mathrm{H}^0(E(-\Delta)) \twoheadrightarrow \mathrm{H}^0(\omega_C \otimes A^{\vee})$. Further, restriction provides an isomorphism $\mathrm{H}^0(X,L) \simeq \mathrm{H}^0(C,\omega_C)$. Thus it suffices to show that the Petri map
$$\mathrm{H}^0(A) \otimes \mathrm{H}^0(\omega_C \otimes A^{\vee}) \to \mathrm{H}^0(\omega_C),$$
is surjective. By Lemma \ref{BNP-gen}, $C$ is Petri general, so the Petri map is injective. But both sides have the same dimension, by the Riemann--Roch theorem, so the Petri map is an isomorphism.

\end{proof}
We now complete the proof of the surjectivity of $\phi$.
\begin{prop}
We have $\mathrm{H}^{k+3}(\bigwedge^2\Gamma (-2k)({p'}^*\Delta)(D))=0$. In particular, $\phi$ is surjective.
\end{prop}
\begin{proof}
We use the exact sequence $0 \to \mathcal{O}_B \to \mathcal{O}_B(D) \to \mathcal{O}_D(D) \to 0.$ We first show \begin{equation} \label{rest-D}
\mathrm{H}^{k+2}(D,\bigwedge^2\Gamma (-2k)({p'}^*\Delta)(2D))=0.
\end{equation}
We have the exact sequence
$$0 \to \mathrm{Sym}^2(S)_{|_D} \to \pi^*\mathcal{M} \otimes S_{|_D} \to \bigwedge^2 \pi^*\mathcal{M}_{|_D} \to \bigwedge^2 \Gamma_{|_D} \to 0.$$
Since $D \to \mathcal{Z}$ is a projective bundle with three dimensional fibres, and $\mathcal{O}_D(D)$ has degree $-1$ on the fibres, $\mathrm{H}^j(D,\pi^*A(nD))=0$
for any bundle $A$ on $X \times \PP$, any integer $j$ and $1 \leq n \leq 3$. Thus $\mathrm{H}^{k+2}(D,\bigwedge^2\pi^*\mathcal{M}(-2k)({p'}^*\Delta)(2D))=0.$ Next, we have $\mathrm{H}^{k+3}(D,\pi^*\mathcal{M} \otimes S(-2k)({p'}^*\Delta)(2D))\simeq \mathrm{H}^{k+3}(D,\pi^*\mathcal{M} \otimes \chi^*\mathcal{H}(-2k-1)({p'}^*\Delta)(2D))$ from the exact sequence $$0 \to \mathcal{O}_D(-2) \to \chi^*\mathcal{H}_{|_D}(-1) \to S_{|_D} \to 0.$$ This space vanishes from the exact sequence $0 \to {p'}^*N \to \chi^*\mathcal{H} \to \mathcal{O}_B(-1)(D) \to 0.$ Lastly, $\mathrm{H}^{k+4}(D, \mathrm{Sym}^2(S)(-2k)({p'}^*\Delta)(2D))=0$ since $\dim D=k+3$, giving the vanishing (\ref{rest-D}).

To finish the proof, it suffices to show $\mathrm{H}^{k+3}(\bigwedge^2\Gamma (-2k)({p'}^*\Delta)(2D))=0$. We have $$\bigwedge^2 \Gamma \simeq \bigwedge^k \Gamma^{\vee}(\det \Gamma)\simeq \bigwedge^k \Gamma^{\vee}(k-1)(-{p'}^*\Delta-D)$$ by Lemma \ref{det-gamma}. Thus, $\mathrm{H}^{k+3}(\bigwedge^2\Gamma (-2k)({p'}^*\Delta)(2D))$ is Serre dual to $\mathrm{H}^{1}(\bigwedge^k\Gamma (2D-2))$. From the exact sequence
{\small{$$0 \to \mathrm{Sym}^k S \to \ldots \to \bigwedge^{k-1}\pi^*\mathcal{M} \otimes S \to \bigwedge^k \pi^* \mathcal{M} \to \bigwedge^k \Gamma \to 0,$$}}
it suffices to show $\mathrm{H}^{1+i}(\bigwedge^{k-i}\pi^*\mathcal{M} \otimes \mathrm{Sym}^i(S)(2D-2))=0$ for $0 \leq i \leq k$. In the case $i=0$ we have $\mathrm{H}^{1}(\bigwedge^{k}\pi^*\mathcal{M}(2D-2))=0$ immediately from the K\"unneth formula on $X \times \PP$. For $1 \leq i \leq k-2$ the required vanishing follows immediately from the exact sequence
$$0 \to \chi^* \mathrm{Sym}^{i-1} \mathcal{H} (-i-1) \to \chi^*\mathrm{Sym}^{i} \mathcal{H} (-i) \to \mathrm{Sym}^i(S) \to 0,$$
as in the proof of Proposition \ref{most-of-van}. \smallskip

It remains to deal with the cases $i=k-1,k$. For $i=k-1$, we have
{\small{
\begin{align*}
\mathrm{H}^{k}(B,\pi^*\mathcal{M} \otimes \mathrm{Sym}^{k-1}(S)(2D-2)) & \simeq \mathrm{H}^{k+1}(B,\pi^*\mathcal{M} \otimes \chi^*\mathrm{Sym}^{k-2}(\mathcal{H})(-k)(2D-2))\\
&\simeq \mathrm{H}^2(\PP(\mathcal{F}),f^*M_L \otimes \mathrm{Sym}^{k-2}(\mathcal{H})\otimes R^{k-1}\chi_*\mathcal{O}_B(-k)\otimes\mathcal{O}_{\PP(\mathcal{F})}(-2))\\
&\simeq \mathrm{H}^2(\PP(\mathcal{F}),f^*M_L \otimes \mathrm{Sym}^{k-2}(\mathcal{H})\otimes\det \mathcal{H}\otimes\mathcal{O}_{\PP(\mathcal{F})}(-2)),
\end{align*}
}}
where we used $\omega_{\chi} \simeq \chi^*\det \mathcal{H}^{\vee}(-k)$ and relative duality for the last equality. Since $\mathrm{Sym}^{k-2}\mathcal{H} \simeq (\chi_*\mathcal{O}_B(k-2))^{\vee} \simeq R^{k-1}\chi_*\mathcal{O}_B(2-2k)\otimes \det \mathcal{H}^{\vee}$, we may identify the above space with $$\mathrm{H}^{k+1}(B, \pi^*\mathcal{M}(-2k)(2D))=0,$$ by the K\"unneth formula.\smallskip

To complete the proof, we need to show $\mathrm{H}^{k+1}(\mathrm{Sym}^k(S)(2D-2))=0$. As in the proof of Proposition \ref{green-odd-main}, the map
$$\mathrm{H}^{k+1}(\mathrm{Sym}^{k-1}(\mathcal{H})(-k-1)(2D-2))\to \mathrm{H}^{k+1}(\mathrm{Sym}^{k}(\mathcal{H})(-k)(2D-2))$$
is surjective, so $\mathrm{H}^{k+1}(\mathrm{Sym}^k(S)(2D-2))$ is isomorphic to 
$$ \mathrm{Ker}\left(  \mathrm{H}^{k+2}(B,\chi^*\mathrm{Sym}^{k-1}\mathcal{H}(-k-1)(2D-2))\twoheadrightarrow \mathrm{H}^{k+2}(B,\chi^*\mathrm{Sym}^{k}\mathcal{H}(-k)(2D-2) \right)$$
which is naturally identified with
{\small{
\begin{align*} \mathrm{Ker}(  &\mathrm{H}^{3}(\PP(\mathcal{F}),\mathrm{Sym}^{k-1}\mathcal{H}\otimes R^{k-1}\chi_*\mathcal{O}_B(-k-1)\otimes \mathcal{O}_{\PP(\mathcal{F})}(-2))\twoheadrightarrow \\
& \mathrm{H}^{3}(\PP(\mathcal{F}), \mathrm{Sym}^{k}\mathcal{H} \otimes R^{k-1}\chi_*\mathcal{O}_B(-k)\otimes \mathcal{O}_{\PP(\mathcal{F})}(-2) ).
\end{align*}
}}
As in the proof of Lemma \ref{identilem}, the above surjection splits and, using relative duality, the above space is identified with 
$$\mathrm{Coker}\left( \mathrm{H}^{k+2}(B,\mathcal{O}_B(-2k-2)(2D)) \to \mathrm{H}^{k+2}(B,\mathcal{O}_B(-2k-1)\otimes \chi^* \mathcal{H}(2D))\right).$$
From the sequence $0 \to {p'}^*N \to \chi^*\mathcal{H} \to \mathcal{O}_B(-1)(D) \to 0$, and since $\mathrm{H}^0(X,N)=\mathrm{H}^1(X,N)=0$, we have $\mathrm{H}^{k+2}(\mathcal{O}_B(-2k-1)\otimes \chi^* \mathcal{H}(2D)) \simeq \mathrm{H}^{k+2}(\mathcal{O}_B(-2k-2)(3D))$. The map 
$ \mathrm{H}^{k+2}(\mathcal{O}_B(-2k-2)(2D)) \to  \mathrm{H}^{k+2}(\mathcal{O}_B(-2k-2)(3D))$ is multiplication by the section $s \in \mathrm{H}^0(\mathcal{O}_B(D))$ induced from the composition $\mathcal{O}_B \hookrightarrow \chi^*\mathcal{H}(1)$, coming from the Euler sequence, with the map $\chi^*\mathcal{H}(1) \twoheadrightarrow \mathcal{O}_B(D)$ coming from the definition of $\mathcal{H}$. Note $s \neq 0$ as $\mathrm{H}^1(B,{p'}^*N(1))=0$, so this map is an isomorphism (by Equation (\ref{mult-D-iso})). This completes the proof.

\end{proof}
\begin{rem}
Note that we do not need the geometry of Grassmannians, unlike \cite[\S 3, Fourth step]{V2}. 
\end{rem}

Let $\psi^{\vee}: \mathrm{Sym}^{k-1}\mathrm{H}^0(E) \xrightarrow{\sim} \mathrm{K}_{k,1}(X,L+\Delta)$ and $\phi^{\vee}: \mathrm{Sym}^{k-1}\mathrm{H}^0(E) \hookrightarrow \mathrm{K}_{k-1,1}(X,-\Delta, L)$ be the duals to the maps from Lemma \ref{green-pic2} and Proposition \ref{green-odd-main}. By Lemma \ref{proj-crit}, to complete the proof of Voisin's Theorem, it only remains to show that $\phi^{\vee}=pr_k \circ \psi^{\vee}$.\smallskip

By duality and from the sequence $0 \to \mathcal{S} \to \pi^* \mathcal{M} \to \Gamma \to 0,$ we identify $\phi^{\vee}$ with
$$\mathrm{H}^{k+2}(B, \wedge^{k-1} \mathcal{S}\otimes \pi^*(L-\Delta \boxtimes \omega_{\PP})(2D)) \to \mathrm{H}^{k+2}(B, \pi^*\wedge^{k-1}M_L(L-\Delta) \boxtimes \omega_{\PP})(2D)),$$
which can, in turn, be identified with the natural map
$$\phi^{\vee} \; : \; \mathrm{H}^{k+2}(X \times \PP, \wedge^{k-1}N (L-\Delta) \boxtimes \omega_{\PP}(1-k)) \to \mathrm{H}^{k+2}(X \times \PP, \wedge^{k-1}M_L(L-\Delta) \boxtimes \omega_{\PP}).$$
\begin{thm}
With notation as above, we have $\mathrm{K}_{k,1}(X,L)=0$.
\end{thm}
\begin{proof}
We need to show $\phi^{\vee}=pr_k \circ \psi^{\vee}$. To relate $\psi^{\vee}$ and $\phi^{\vee}$, let $\tilde{\pi}: \widetilde{B} \to X \times \PP$ denote the blow-up in the codimension two locus $\widetilde{\mathcal{Z}}:=\{ (x,s) \; | s(x)=0 \}$. Let $\widetilde{D}$ denote the exceptional divisor. Define $\tilde{p}:=p \circ \tilde{\pi}$ and $\tilde{q}:=q \circ \tilde{\pi}$. We have the exact sequence
$$0 \to \tilde{\pi}^*(\mathcal{O}_X(-\Delta) \boxtimes \mathcal{O}_{\PP}(-2))(\widetilde{D}) \to \tilde{\pi}^*(E(-\Delta) \boxtimes \mathcal{O}_{\PP}(-1)) \to {\tilde{p}}^* L \otimes I_{\widetilde{D}} \to 0.$$
We have $\tilde{q}_*({\tilde{p}}^* L \otimes I_{\widetilde{D}}) \simeq \mathrm{H}^0(E(-\Delta)) \otimes \mathcal{O}_{\PP}(-1)$. Define $\tilde{S}_{L'}$ and $\tilde{S}_L$ by exact sequences
{\small{\begin{align*}
0 \to &\tilde{S}_{L'} \to {\tilde{q}}^*{\tilde{q}}_*({\tilde{p}}^* {L'} \otimes I_{\widetilde{D}})  \to {\tilde{p}}^* L' \otimes I_{\widetilde{D}} \to 0, \\
0 \to &\tilde{S}_L \to \mathrm{H}^0(E(-\Delta)) \otimes {\tilde{q}}^*\mathcal{O}_{\PP}(-1) \to {\tilde{p}}^* L \otimes I_{\widetilde{D}} \to 0.
\end{align*}}}
Then $\psi^{\vee}$ is the natural map $\psi^{\vee} \; : \; \mathrm{H}^{k+3}(\wedge^{k+1}\tilde{S}_{L'} \otimes \omega_{\widetilde{B}}) \to \mathrm{H}^{k+3}(\tilde{\pi}^*(\wedge^{k+1}M_{L'} \boxtimes \mathcal{O}_{\PP}) \otimes \omega_{\widetilde{B}}).$
By taking exterior powers of the exact sequence
$$0 \to \tilde{\pi}^*(N \boxtimes \mathcal{O}_{\PP}(-1)) \to \tilde{S}_L \to \tilde{\pi}^* (\mathcal{O}_X(-\Delta) \boxtimes \mathcal{O}_{\PP}(-2))(\widetilde{D}) \to 0$$
and using $\mathrm{H}^0(X,\wedge^{k-2}N(L-2\Delta))=\mathrm{H}^0(N^{\vee}(-\Delta))=\mathrm{H}^2(N(\Delta))=0$, identify $\phi^{\vee}$ with
$$\mathrm{H}^{k+2}(\wedge^{k-1} \tilde{S}_L (L-\Delta) \otimes {\tilde{q}}^* \omega_{\PP}) \to \mathrm{H}^{k+2}(\tilde{\pi}^*\wedge^{k-1}M_L (L-\Delta) \boxtimes \omega_{\PP}),$$
induced by $\tilde{S}_L \hookrightarrow \tilde{p}^* M_L$. Using exterior powers of the defining sequence for $\tilde{S}_L$, this can be further identified with
$$\phi^{\vee} \; : \; \mathrm{H}^{k+3}(\wedge^k \tilde{S}_L(-{\tilde{p}}^*\Delta) \otimes \omega_{\widetilde{B}}) \to \mathrm{H}^{k+3}(\tilde{\pi}^*(\wedge^k M_L(-\Delta) \boxtimes \mathcal{O}_{\PP}) \otimes \omega_{\widetilde{B}}),$$
using $\mathrm{H}^0(\mathcal{O}_X(-\Delta))=\mathrm{H}^1(\mathcal{O}_X(-\Delta))=0$.\smallskip

Let $U \seq \widetilde{B}$ be the complement of the codimension two locus $\tilde{\pi}^{-1}(X \times \PP(\mathrm{H}^0(E(-\Delta)))$. We have an exact sequence
$0 \to \tilde{S}_{L_{|_U}} \to \tilde{S}_{L'_{|_U}} \to \mathcal{O}_U(-\tilde{p}^*\Delta) \to 0$, giving a commutative diagram
{\small{
$$\begin{tikzcd}
\wedge^{k+1} \tilde{S}_{L'_{|_U}} \arrow[d, "\simeq"] \arrow[r] & \tilde{\pi}^*(\wedge^{k+1}M_{L'} \boxtimes \mathcal{O}_{\PP})_{|_U} \arrow[d, "q_{|_U}"] \\
\wedge^k \tilde{S}_{L_{|_U}}( -\tilde{p}^*\Delta)  \arrow[r] & \tilde{\pi}^*( \wedge^k M_L(-\Delta) \boxtimes \mathcal{O}_{\PP})_{|_U} 
\end{tikzcd}$$
}}
where $q: \tilde{\pi}^*(\wedge^{k+1}M_{L'} \boxtimes \mathcal{O}_{\PP}) \to \tilde{\pi}^*(\wedge^k M_L(-\Delta) \boxtimes \mathcal{O}_{\PP})$ is the projection. This diagram extends uniquely to $\widetilde{B}$, giving the claim. \end{proof}

\end{document}